\documentclass[a4paper,fleqn,reqno,12pt]{amsart}

\usepackage[utf8]{inputenc}
\usepackage[in]{fullpage}

\usepackage{amsmath,amssymb,amsthm}
\usepackage[alphabetic,nobysame]{amsrefs}
\usepackage[foot]{amsaddr}
\input amsart-modif

\usepackage[colorlinks, breaklinks, allcolors=blue]{hyperref}

\usepackage{parskip}

\numberwithin{equation}{section} \swapnumbers

\usepackage{tikz}
\usetikzlibrary{calc,arrows}
\usepackage[cmtip,matrix,arrow]{xy}
\SelectTips{cm}{10}

\usepackage{enumitem}
\setlist[enumerate,1]{label=\textit{\alph*)},ref=\textit{\alph*})}
\setlist[enumerate,2]{label=\textit{\roman*)},ref=\textit{\roman*})}

\usepackage{aliascnt}

\newtheorem{theorem}{Theorem}[section]

\newaliascnt{lemma}{theorem}
\newtheorem{lemma}[lemma]{Lemma}
\aliascntresetthe{lemma}

\newaliascnt{corollary}{theorem}
\newtheorem{corollary}[corollary]{Corollary}
\aliascntresetthe{corollary}

\newaliascnt{proposition}{theorem}
\newtheorem{proposition}[proposition]{Proposition}
\aliascntresetthe{proposition}

\theoremstyle{definition}
\newaliascnt{definition}{theorem}
\newtheorem{definition}[definition]{Definition}
\aliascntresetthe{definition}

\newtheorem*{ack}{Acknowledgment}

\newtheorem*{remark}{Remark}
\newtheorem*{remarks}{Remarks}

\usepackage[capitalize]{cleveref}

\newcommand{\cxymatrix}[1]{\vcenter{\xymatrix@=15pt{#1}}}

\def\Ddot(#1,#2;#3){\filldraw (#1,#2) node[above]{\hbox to
    0pt{\hss$#3$\hss}} circle (3pt);}

\def\Dline(#1,#2){\draw[thick] (#1,#2)--+(1,0);}

\def\Drarrow(#1,#2){%
  \draw[thick] (#1,#2+1pt)--+(0.85,0);
  \draw[thick] (#1,#2-1pt)--+(0.85,0);
  \draw[thick] (#1+0.9,#2)--+(-8pt,4pt);
  \draw[thick] (#1+0.9,#2)--+(-8pt,-4pt); }

\def\Dlarrow(#1,#2){%
  \draw[thick] (#1,#2+1pt)--+(-0.85,0);
  \draw[thick] (#1,#2-1pt)--+(-0.85,0);
  \draw[thick] (#1-0.9,#2)--+(8pt,4pt);
  \draw[thick] (#1-0.9,#2)--+(8pt,-4pt); }

\def\Dlrarrow(#1,#2){%
  \draw[thick] (#1+0.15,#2+1pt)--+(0.7,0);
  \draw[thick] (#1+0.15,#2-1pt)--+(0.7,0);
  \draw[thick] (#1+0.9,#2)--+(-8pt,4pt);
  \draw[thick] (#1+0.9,#2)--+(-8pt,-4pt);
  \draw[thick] (#1+0.1,#2)--+(8pt,4pt);
  \draw[thick] (#1+0.1,#2)--+(8pt,-4pt); }

\def\Ddots(#1,#2){%
  \draw[thick] (#1,#2)--+(0.6,0);
  \filldraw (#1+0.8,#2) circle (0.5pt);
  \filldraw (#1+1,#2) circle (0.5pt);
  \filldraw (#1+1.2,#2) circle (0.5pt);
  \draw[thick] (#1+1.4,#2)--(#1+2,#2); }

\newcommand{\cA}{\mathcal{A}}
\newcommand{\cC}{\mathcal{C}}
\newcommand{\cX}{\mathcal{X}}
\newcommand{\cF}{\mathcal{F}}
\newcommand{\cL}{\mathcal{L}}
\newcommand{\cO}{\mathcal{O}}
\newcommand{\cP}{\mathcal{P}}

\newcommand{\fa}{\mathfrak{a}}
\newcommand{\fL}{\mathfrak{L}}
\newcommand{\fk}{\mathfrak{k}}
\newcommand{\fl}{\mathfrak{l}}
\newcommand{\fp}{\mathfrak{p}}
\newcommand{\Fq}{\mathfrak{q}}
\newcommand{\fK}{\mathfrak{K}}
\newcommand{\ft}{\mathfrak{t}}
\newcommand{\fC}{\mathfrak{C}}
\newcommand{\fT}{\mathfrak{T}}

\newcommand{\fS}{\mathfrak{S}}

\newcommand{\CC}{\mathbb{C}}
\newcommand{\HH}{\mathbb{H}}
\newcommand{\RR}{\mathbb{R}}
\newcommand{\ZZ}{\mathbb{Z}}

\newcommand{\sA}{\mathsf{A}}
\newcommand{\sB}{\mathsf{B}}

\newcommand{\sD}{\mathsf{D}}
\newcommand{\sG}{\mathsf{G}}

\renewcommand{\P}{\mathbf{P}}

\newcommand{\Vfa}{{\overline\fa}}
\newcommand{\fq}{{\overline f}}
\newcommand{\Valpha}{{\overline\alpha}}
\newcommand{\Vbeta}{{\overline\beta}}
\newcommand{\Vw}{{\overline w}}
\newcommand{\VW}{{\overline W}}
\newcommand{\VPhi}{{\overline\Phi}}
\newcommand{\oS}{{\overline S}}
\newcommand{\thetaq}{{\overline\theta}}
\newcommand{\tauq}{{\overline\tau}}
\newcommand{\uq}{{\overline u}}

\renewcommand{\[}{\begin{equation}}
\renewcommand{\]}{\end{equation}}
\newcommand{\<}{\langle}
\renewcommand{\>}{\rangle}

\renewcommand{\epsilon}{\varepsilon}
\renewcommand{\rho}{\varrho}
\renewcommand{\mod}{/\!\!/}
\renewcommand{\phi}{\varphi}
\renewcommand{\hat}{\widehat}
\newcommand{\leer}{\varnothing}
\newcommand{\into}{\hookrightarrow}
\newcommand{\Pfeil}{\longrightarrow}
\newcommand\otau{\,{}^\tau\!}
\newcommand{\qH}{quasi-Hamiltonian}
\newcommand{\mf}{multiplicity free quasi-Hamiltonian}
\newcommand{\MF}{{\mathsf{MF}}}
\newcommand\Samb{S_{\rm amb}}
\newcommand{\Times}{\mathop{\times}\limits}
\newcommand{\gl}{_{\|gl|}}
\newcommand{\half}{{\textstyle\frac12}}
\newcommand{\viertel}{{\textstyle\frac14}}
\newcommand{\OMEGA}{\cL^1}
\newcommand{\prim}{^{\text{prim}}}

\DeclareMathOperator{\OG}{O}
\DeclareMathOperator{\Ad}{Ad}
\DeclareMathOperator{\Hom}{Hom}
\DeclareMathOperator{\Aut}{Aut}
\DeclareMathOperator{\cAut}{\underline{\mathrm{Aut}}}
\DeclareMathOperator{\res}{res}
\DeclareMathOperator{\Lie}{Lie}
\DeclareMathOperator{\rk}{rk}
\DeclareMathOperator{\pr}{pr}
\DeclareMathOperator{\id}{id}
\DeclareMathOperator{\Gr}{Gr}

\let\norm=\|
\def\|#1|{\operatorname{#1}}

\title[]{Multiplicity free quasi-Hamiltonian manifolds}

\author[]{Friedrich Knop\newline FAU Erlangen-Nürnberg}
\address[]{Department Mathematik, Emmy-Noether-Zentrum, Cauerstraße
  11, 91058 Erlangen, Germany}

\usepackage{microtype}

\begin{document}

\setlength\mathindent{60pt}

\begin{abstract}

  A (quasi-)Hamiltonian manifold is called multiplicity free if all
  of its symplectic reductions are 0-dimensional. In this paper, we
  classify multiplicity free Hamiltonian actions for (twisted) loop
  groups or, equivalently, multiplicity free (twisted) \qH\ manifolds
  for simply connected compact Lie groups. As a result we recover
  old and find new examples of these structures.

\end{abstract}

\maketitle

\section{Introduction}\label{Intro}

In \cite{AMM}, Alekseev-Malkin-Meinrenken introduced the concept of a
\emph{group valued moment map} and showed that it is essentially
equivalent to the concept of Hamiltonian loop groups actions. The
advantage of the former is that all objects are finite dimensional. As
opposed to the ordinary moment map which takes values in the coadjoint
representation of the acting Lie group $K$, a group valued moment map
takes values in $K$ itself (which, in this paper, is usually assumed
to be compact and simply connected). A manifold equipped with a group
valued moment map is called \emph\qH.

Like Hamiltonian manifolds, \qH\ manifolds also have a notion of
symplectic reduction and the dimension of the symplectic reductions of
a \qH\ manifold serves as a measure for its size. At the bottom of
this hierarchy lie the so called \emph{multiplicity free} (\qH)
manifolds, i.e., those manifolds having $0$-dimensional
reductions. The purpose of the present paper is to classify \mf\
manifolds under two technical conditions: the group $K$ should be
simply connected and the manifold should be convex. The first
condition is essential while the second is rather mild and holds
automatically for compact manifolds.

Actually, it turned out to be more natural to consider also
\emph{twisted} \qH\ manifolds. This means that $K$ is equipped with an
automorphism $\tau$ such that the moment map to $K$ is equivariant
with respect to twisted conjugation $g\mapsto kg\tau(k)^{-1}$. In
analogy with \cite{AMM}, twisted \qH\ manifolds correspond to
Hamiltonian actions of twisted loop groups. Here a twisted loop is a
map $g:\RR\to K$ satisfying the periodicity condition
$g(t+1)=\tau(g(t))$.

The classification is in terms of two data. The first encodes the
image of the moment map $m:M\to K$ . For this recall that the set of
(twisted) conjugacy classes of $K$ is in bijection with an alcove
$\cA$ for a certain affine root system (see, e.g.,
\cite{MW}). Therefore, the image of $m$ is determined by a subset
$\cP_M\subseteq\cA$ which turns out to be a convex polyhedron, if $M$
is compact and which is, in general, locally polyhedral. The second
datum is a lattice $\Lambda_M$ which encodes the principal isotropy
group of $K$ on $M$. These two objects $\cP_M$ and $\Lambda_M$ satisfy
certain compatibility conditions called sphericality (see below for
more) and our main result, \cref{thm:main}, says that convex \mf\
manifolds are classified by spherical pairs $(\cP,\Lambda)$. Observe
that the uniqueness part is a \qH\ generalization of the Delzant
conjecture \cite{Delzant} (which is proved in \cite{KnopAuto}).

The main application of our classification is the construction of \qH\
manifolds. This is much harder than in the Hamiltonian setting since
given a subgroup $H\subseteq K$ it is, in general, not possible to
restrict a $K$-valued moment map to an $H$-valued one. So our approach
is to construct spherical pairs $(\cP,\Lambda)$ which then lead to
\qH\ manifolds.

Using this technique, we are able to recover most examples which have
been previously constructed ``by hand'': The double of a group by
Alekseev-Malkin-Meinrenken, \cite{AMM}, the spinning $4$-sphere by
Alekseev-Meinrenken-Woodward, \cite{AMW}, its generalization, the
spinning $2n$-sphere by Hurtubise-Jeffreys-Sjamaar, \cite{HJS}, and
the quaternionic projective space due to Eshmatov, \cite{Eshmatov}.

Beyond that, we show that more generally the quaternionic
Grassmannians $\Gr_k(\HH^{n+1})$ carry a \qH\ $Sp(2n)$-structure. We
also find compact \qH\ manifolds for the groups $SU(n)$ and $Sp(2n)$,
respectively, for which the moment map is surjective, something which
has no analogue for Hamiltonian manifolds. On the side we observe that
the product of any two symmetric spaces for the same group with
diagonal action (we call them \emph{disymmetric}) carries a \mf\
structure. This explains many of their nice invariant theoretic
properties.

Some words on the sphericality condition. Recall that a (complex
algebraic) variety $X$ with an action of a connected reductive group
$G$ is called \emph{spherical} if a Borel subgroup of $G$ has a dense
open orbit in $X$. When $X$ is affine this has a purely representation
theoretic interpretation: $X$ is spherical if and only if its
coordinate ring $\CC[X]$ is a multiplicity free $G$-module. In that
case, $\CC[X]$ is, as a $G$-module, uniquely determined by the set
$\Lambda^+_X$ of highest weights occurring in it. If $X$ is
additionally smooth then it is even uniquely determined by
$\Lambda_X^+$ (Losev \cite{Losev}).

Now it follows from work of Brion \cite{Brion} and Sjamaar
\cite{Sjamaar} that a multiplicity free manifold $M$ is locally
modeled after a smooth affine spherical variety. This means the
following: For any $x\in\cP_M$ let $L\subseteq K$ be the (twisted)
centralizer of $\exp(x)$ in $K$. Then an open $K$-invariant
neighborhood of $m^{-1}(x)$ in $M$ is isomorphic to an open subset of
a ``model space'' of the form $K\times^LX$ where $X$ is a smooth
affine spherical $L_\CC$-variety. Strictly speaking, Brion and Sjamaar
proved this only for ordinary Hamiltonian manifolds but using
techniques from \cite{AMM} it readily generalizes to the twisted \qH\
setting.

This has the following consequence for the pair
$(\cP,\Lambda):=(\cP_M,\Lambda_M)$. Let $C_X$ be the convex cone and
$\Lambda_X$ the abelian group $\Lambda_X$ generated by
$\Lambda_M^+$. Then
\[\label{eq:sph}
  C_x\cP=C_X\text{ and }\Lambda=\Lambda_X
\]
where $C_x\cP=\RR_{\ge0}(\cP-x)$ is the tangent cone of $\cP$ in
$x$. Conversely, we call a pair $(\cP,\Lambda)$ spherical if for every
$x\in\cP$ there is a smooth affine spherical $L_\CC$-variety $X$ such
that \eqref{eq:sph} holds.

Admittedly, the sphericality condition is not very explicit since it
involves finding an appropriate variety $X$ for every point
$x\in\cP$. This task is simplified by two facts: First of all, it
suffices to check it for representatives of the minimal faces of
$\cP$. So if $\cP$ is a polyhedron then it is enough to check the
vertices. Secondly, Van Steirteghem and the author have essentially
classified all smooth affine spherical varieties in \cite{KVS}. A
description of their weigh monoids will appear in joint work with
Pezzini and Van Steirteghem \cite{KPVS}. These works make it possible
to test sphericality in any given vertex. Conversely, one can use the
classification to look for local models such that the tangent cones
and the lattice (according to \eqref{eq:sph}) paste to a global
spherical pair. See section \ref{sec:examples} for examples on how
this strategy works.

This paper is, to a certain extent, a sequel of \cite{KnopAuto} where
analogous results were proved in the Hamiltonian
setting. Nevertheless, I have tried to make the present paper
self-contained enough such that at least the main results should be
understandable without consulting \cite{KnopAuto} or even
\cite{AMM}. On the other hand, for some of the main arguments,
especially the local structure of multiplicity free Hamiltonian
manifolds and their automorphism groups, we refer to
\cite{KnopAuto}. The cohomology computations in section
\ref{sec:cohomology} are quite a bit more involved for \qH\ manifolds
than those of \cite{KnopAuto}.

We include systematically the twisted case and most multiplicity free
examples in this setting seem to be new. After completion of a first
version of this paper, the author became aware of Meinrenken's
manuscript \cite{Meinrenken} which explicitly studies twisted \qH\
manifolds. Therefore, it has some overlap with sections
\ref{sec:loop}, \ref{sec:twisted}, and \ref{sec:localstructure}.

Finally some advice for reading the paper: I gathered all examples in
the final section \ref{sec:examples} but most of them can be
understood much sooner.

\begin{ack}

  I would like to thank Chris Woodward who, eons ago, suggested the
  topic of this paper to me. Thanks are also due to Kay Paulus and
  Bart Van Steirteghem for numerous discussions about this paper.

\end{ack}

\section{From Hamiltonian loop group spaces to \qH\ manifolds}
\label{sec:loop}

In this section, we give a brief introduction to Hamiltonian actions
of twisted loop groups. The untwisted case has been worked out by
Alekseev, Malkin and Meinrenken in their nice paper \cite{AMM}. For a
short survey see also \S1.4 of \cite{GS}. Therefore, the main purpose
of this and the following section is to make precise where to put the
twist $\tau$ in.

In the whole paper, let $K$ be a compact connected Lie group with Lie
algebra $\fk$. Let $\tau$ be a fixed ``twist'' of $K$ which simply
means a continuous automorphism of $K$. The action of $\tau$ on
$k\in K$ is going to be denoted by ${}^\tau k$. The induced
automorphism of $\fk$ will also be denoted by $\tau$.

We also fix a $K$- and $\tau$-invariant scalar product
$\<\cdot,\cdot\>$ on $\fk$. For semisimple $K$ one could take the
Killing form but it will be useful to also consider differently scaled
scalar products (see, e.g., the final remark of example~2 in section
\ref{sec:examples}).

The \emph{twisted loop group $\cL_\tau(K)$} is the set of smooth maps
$g:\RR\to K$ which are subject to the condition
\[
  g(t+1)={}^\tau g(t)\text{ for all $t\in\RR$}.
\]
It is a group under pointwise multiplication. Clearly, if $\tau=\id_K$
then $\cL_\tau(K)$ is just the group of smooth loops
$g:S^1=\RR/\ZZ\to K$.

The space $\cL_\tau(\fk)$ of smooth maps $\xi:\RR\to\fk$ with
$\xi(t+1)={}^\tau\xi(t)$ is the Lie algebra of $\cL_\tau(K)$. The
invariant scalar product on $\fk$ induces a scalar product on
$\cL_\tau(\fk)$ by
\[
  \<\xi,\eta\>_\cL:=\int_0^1\<\xi(t),\eta(t)\>dt.
\]
Observe that the integrand is periodic of period $1$, hence
integration over any interval of length $1$ yields the same result. It
also follows that
\[
  c(\xi,\eta):=\<\xi',\eta\>_\cL=-\<\xi,\eta'\>_\cL.
\]
(where $\xi'$, $\eta'$ are the derivatives with respect to $t$) is a
$2$-cocycle and therefore defines a central extension
$\hat\cL_\tau(\fk):=\cL_\tau(\fk)\oplus\RR \kappa$ of $\cL_\tau(\fk)$
by
\[ [\xi+s\kappa,\eta+t\kappa]:=[\xi,\eta]+c(\xi,\eta)\kappa.
\]
Dually, the space $\hat\cL_\tau(\fk^*):=\cL_\tau(\fk)\oplus\RR E$ is
in duality with $\hat\cL_\tau(\fk)$ by
\[
  \<A+sE,\xi+t\kappa\>_\cL:=\<A,\xi\>_\cL+st.
\]
The dual (coadjoint) action of $\cL_\tau(\fk)$ on
$\hat\cL_\tau(\fk^*)$ is then
\[
  \xi\cdot(A+sE)=[\xi,A]-t\xi'.
\]
Therefore the action of $g(t)\in\cL_\tau(K)$ on
$A(t)+sE\in\hat\cL_\tau(\fk^*)$ is
\[
  g\cdot(A+sE)=\left(\Ad(g)A-sg'\,g^{-1}\right)+sE.
\]
Since the so-called \emph{level} $s=\<\cdot,\kappa\>_\cL$ is invariant
under this action, the group $\cL_\tau(K)$ acts on the level-$1$-set
$\OMEGA_\tau(\fk):=\cL_\tau(\fk)+E$. Under the identification
$\cL_\tau(\fk)\cong\OMEGA_\tau(\fk):A\mapsto A+E$ this action becomes
\[\label{eq:action}
  g\cdot A=\Ad(g)A-g'\,g^{-1}.
\]

\begin{definition}

  A \emph{Hamiltonian $\cL_\tau(K)$-space} (of level $1$) is a Fréchet
  manifold $X$ equipped with an $\cL_\tau(K)$-action, a $2$-form
  $\sigma$, and map $\mu:X\to\OMEGA_\tau(\fk)$ such that

  \begin{enumerate}

  \item The map $\mu$ is smooth and equivariant with respect to the
    action \eqref{eq:action}.

  \item The $2$-form $\sigma$ is $\cL_\tau(K)$-invariant, closed and
    non-degenerate

  \item $\sigma(\xi x,\eta)=\<\xi,\mu_*\eta\>_\cL$ \text{ for all
      $\xi\in\cL_\tau(\fk)$ and $\eta\in T_xX$}

  \end{enumerate}

  The Hamiltonian space $X$ is of \emph{finite type} if $\mu$ is
  locally proper, i.e., every $x\in X$ has a (closed) neighborhood $U$
  such that $\mu|_U:U\to\OMEGA_\tau(\fk)$ is proper.

\end{definition}

\begin{remarks} \textit{a)} Because of its importance for geometry, it
  is customary to only consider a level of $1$. Clearly, by rescaling
  this encompasses also actions of any positive but constant level. As
  a matter of fact, it appears that the whole theory of this paper
  should generalize to actions of non-constant level as long as it
  stays positive, i.e., to Hamiltonian $\hat\cL_\tau(K)$-spaces with
  $\<\mu(x),\kappa\>_\cL>0$ for all $x\in X$. Clearly, in such
  generality the center of $\hat\cL_\tau(K)$ will act non-trivially on
  $X$.

  \textit{b)} The $\cL_\tau(\fk)$-equivariance of $m$ is equivalent to
  the formula
  \[
    \sigma(\xi x,\eta x)=\<[\xi,\eta],\mu(x)\>_\cL+c(\xi,\eta)\text{
      for all }x\in X\text{ and }\xi,\eta\in\cL_\tau(\fk).
  \]

\end{remarks}

Before we go on, we recall (and twist) the basic facts of the
coadjoint loop group action. The key is the observation that any
$A\in\OMEGA_\tau(\fk)$ defines a connection on the trivial $K$-bundle
$p:K\times\RR\to\RR$ (with $K$ acting on the left). Then the gauge
group $\cL_\tau(L)$ acts on the right of $K\times\RR$. The action on
connections is given by equation \eqref{eq:action}. More precisely, a
horizontal section of the connection corresponding to $A$ is a
solution $z:\RR\to K$ of the ordinary differential equation
\[\label{eq:ODE}
  z'(t)=z(t)A(t).
\]
Any other solution is of the form $kz(t)$ with $k\in K$. In
particular,
\[
  h_t(A):=z(0)^{-1}z(t)
\]
depends only on $A$ and not on the choice of $z$. The value at $t=1$
\[
  h(A):=h_1(A)=z(0)^{-1}z(1)\in K
\]
is called the \emph{holonomy of $A$}.  It is easily checked that if
$g(t)\in\cL_\tau(K)$ then $\tilde z(t):=z(t)g(t)^{-1}$ is a solution
of \eqref{eq:ODE} for $\tilde A:=g\cdot A$ which implies
\[\label{eq:taction}
  h_t(g\cdot A)=\tilde z(0)^{-1}\tilde
  z(t)=g(0)z(0)^{-1}z(t)g(t)^{-1}= g(0)\,h_t(A)\,g(t)^{-1}.
\]
In particular, if we put $t=1$ and observe that $g(1)=\otau g(0)$ we
get
\[\label{eq:twistedaction}
  h(g\cdot A)=g(0)\,h(A)\,\otau g(0)^{-1}.
\]

The evaluation map $g\mapsto g(0)$ induces the short exact sequence
\[
  1\mapsto\Omega_\tau(K)\to\cL_\tau(K)\overset\to K\to 1.
\]
where $\Omega_\tau(K)$ is the group of based (twisted) loops, i.e.,
with $g(0)=1$. Observe that, unlike in the untwisted case, the
evaluation homomorphism does not have a canonical section. Then
\eqref{eq:twistedaction} says that $h$ is $\cL_\tau(K)$-equivariant
where $K$ acts on itself by the twisted action
\[
  g\cdot_\tau k:=gk\otau g^{-1}.
\]
To distinguish the twisted action from the adjoint action we are going
to write $K\tau$ for $K$ when we mean the former. This makes sense
since, purely formally, we have
\[
  \Ad(g)(k\tau)=gk\tau g^{-1}=gk(\tau g^{-1}\tau^{-1})\tau=(gk\otau
  g^{-1})\tau.
\]
Of course, this calculation can be made rigorous by considering
$K\tau$ as a subset of the group $\ZZ\tau\ltimes K$.

The following fact is fundamental:

\begin{lemma}

  The map $h:\OMEGA_\tau(\fk)\to K\tau$ is an
  $\cL_\tau(K)$-equivariant principal fiber bundle for
  $\Omega_\tau(K)$.

\end{lemma}

\begin{proof}

  Consider a solution $z$ of \eqref{eq:ODE}. Applying $\tau$ yields
  that both $z(t+1)$ and $\otau z(t)$ are horizontal sections for
  $\otau A$ which implies
  \[\label{eq:zt+1}
    z(t+1)=k\otau z(t)\text{ where }k:=z(0)h(A)\otau z(0)^{-1}.
  \]
  Conversely, for every fixed element $k_0\in K$ there is clearly a
  smooth map $z_0:\RR\to K$ with $z_0(t+1)=k\otau z_0(t)$ and
  $z_0(0)=1$. Then $A_0:=z_0(t)^{-1}z_0'(t)$ is an element of
  $\OMEGA_\tau(\fk)$ with $h(A_0)=k_0$. This shows that $h$ is
  surjective.

  Now let $A_1\in\OMEGA_\tau(\fk)$ be a second element with
  $h(A_1)=k_0$. Let $z_1(t)$ be the corresponding horizontal section
  with $z_1(0)=1$. Then \eqref{eq:zt+1} implies that
  $g(t):=z_1(t)^{-1}z_0(t)$ is an element of
  $\Omega_\tau(\fk)$. Moreover, an easy calculation shows
  $g\cdot A_1=A_2$. Hence $\Omega_\tau(K)$ acts transitively on the
  fiber $h^{-1}(k_0)$. On the other hand, let $g(t)\in\Omega_\tau(K)$
  with $g\cdot A_0=A_0$. Then $\tilde z:=z_0g^{-1}$ is also a
  horizontal section for $A$ with $z_0(0)=1$. Hence, $\tilde z=z_0$
  and therefore $g\equiv1$ which means that the action of
  $\Omega_\tau(K)$ is free.

  Finally, the bundle $h$ is locally free since $z_0$ can be chosen to
  depend smoothly on $k_0$ in a small open subset.
\end{proof}

We return to a Hamiltonian space $X$ with moment map
$\mu:X\to\OMEGA_\tau(\fk)$.  Since $\Omega_\tau(K)$ acts freely on the
target, its action on $X$ is free, as well. Let
$\tilde h:X\to M:=X/\Omega_\tau(K)$ be the quotient and
$m:=\mu/\Omega_\tau(\fk):M\to K$. Then clearly, the following square
is Cartesian:
\[
  \cxymatrix{X\ar[r]^\mu\ar@{>>}[d]^{\tilde
      h}&\OMEGA_\tau(\fk)\ar@{>>}[d]^h
    \\
    M\ar[r]^m&K\tau}
\]
Hence $X$ and $\mu$ can be reconstructed from $M$ and $m$. To also get
the $2$-form $\sigma$, Alekseev-Malkin-Meinrenken introduce in
\cite{AMM} additional structure on $M$.

For this, let $\theta$ and $\thetaq$ be the two canonical $\fk$-valued
$1$-forms on $K$, defined by $\theta(k\xi)=\xi=\thetaq(\xi k)$. These
are combined to get another $\fk$-valued $1$-form
\[
  \Theta_\tau:=\half\big(\,\thetaq+{}^{\tau^{-1}}\theta\,\big)
\]
with $\tau$ acting on the target $\fk$. Thus
$\Theta_\tau(k\xi)=\half(\Ad(k)\xi+{}^{\tau^{-1}}\xi)$.

Moreover, the scalar product on $\fk$ induces the canonical
biinvariant closed $3$-form on $K$
\[
  \chi:=\frac1{12}\<\theta,[\theta,\theta]\>=
  \frac1{12}\<\thetaq,[\thetaq,\thetaq]\>.
\]

\begin{definition}\label{D1}

  A {\it \qH\ $K\tau$-manifold} is a smooth manifold $M$ equipped with
  a $K$-action, a $2$-form $\omega$, and a smooth map $m:M\to K\tau$,
  called the {\it (group valued) moment map}, having the following
  properties:

  \begin{enumerate}

  \item\label{D1i0} $m$ is $K$-equivariant.

  \item\label{D1i1} The form $\omega$ is $K$-invariant and satisfies
    $\mathrm{d}\omega=-m^*\chi$.

  \item\label{D1i3} $\omega(\xi x,\eta)=\<\xi,m^*\Theta_\tau(\eta)\>$
    for all $\xi\in\fk$ and $\eta\in T_xM$.

  \item\label{D1i2}
    $\ker\omega_x=\{\xi x\in T_xM\mid\xi\in\fk\text{ with }
    {}^{m(x)\tau}\xi+\xi=0\}$.

  \end{enumerate}

\end{definition}

To get the connection with $X$, consider the family of $\fk$-valued
$1$-forms $\Xi_t:=h_t^*\thetaq$ on $\OMEGA_\tau(\fk)$ and the $2$-form
\[
  \varpi:=\half\int_0^1\<\Xi_t,\frac{d\Xi_t}{dt}\>\ dt.
\]

One of the main results of \cite{AMM} is:

\begin{theorem}

  Let $(X,\sigma,\mu)$ be a Hamiltonian $\cL_\tau(K)$-space which is
  of finite type.

  \begin{enumerate}

  \item Then there is a unique $2$-form $\omega$ on
    $M=X/\Omega_\tau(K)$ such that
    $\sigma+\mu^*\varpi=\tilde h^*\omega$.

  \item The triple $(M,\omega,m)$ is a \qH\ $K\tau$-manifold.

  \end{enumerate}

  Moreover, the functor $(X,\sigma,\mu)\mapsto(M,\omega,m)$ is an
  equivalence of categories between Hamiltonian $\cL_\tau(K)$-spaces
  and quasi-Hamiltonian $K\tau$-manifolds (both with
  morphisms=iso\-mor\-phisms).

\end{theorem}

\begin{proof}

  The quasi-inverse functor is the fiber product
  $X=M\times_K\OMEGA_\tau(\fk)$. In the untwisted case, this has been
  proved in \cite{AMM}. The main point is the derivation formula for
  $\iota_{v_\xi}\varpi$ in \cite{AMM}*{Appendix A} which now yields
  \begin{multline}
    \iota_{v_\xi}\varpi=-d_A\<A,\xi\>_\cL+\frac12\<h^*\thetaq,\xi(0)\>+
    \frac12\<h^*\theta,\xi(1)\>=\\=
    -d_A\<A,\xi\>_\cL+\<h^*\Theta_\tau,\xi(0)\>.
  \end{multline}
  Here, $v_\xi$ is the vector field on $\OMEGA_\tau(\fk)$ which is
  induced by the action given by \eqref{eq:action} of
  $\xi\in\cL_\tau(\fk)$.
\end{proof}

Let's call a Hamiltonian $\cL_\tau(K)$-space $X$ {\it complete} if its
moment map $\mu:X\to\OMEGA_\tau(\fk)$ is proper. This leads to the
following observation:

\begin{corollary}

  Let $M$ be the \qH\ $K\tau$-manifold associated to the Hamiltonian
  $\cL_\tau(K)$-space $X$. Then $X$ is complete if and only if $M$ is
  compact.

\end{corollary}

Now we are able to introduce the main objects of the present paper. To
motivate it, consider one of the most important operations for
Hamiltonian manifolds namely {\it symplectic reduction}. Assume that
$A\in\OMEGA_\tau(\fk)$ is in the image of the moment map. Then
\[
  X_A:=\mu^{-1}(A)/\cL_\tau(K)_A=\mu^{-1}(\cL_\tau(K)\cdot
  A)/\cL_\tau(K)
\]
is called the symplectic reduction of $X$ at $A$. Clearly, it depends
only on the coadjoint orbit defined by $A$. Another way to describe it
is to consider the image $a:=h(A)\in K$. Then $X_A=M_a$ where
\[
  M_a:=m^{-1}(a)/K_a=m^{-1}(Ka)/K.
\]
It is known that for general $a$, the space $M_a$ is a symplectic
manifold whose (even) dimension is independent of $a$. So we define
the \emph{complexity} of $X$ or $M$ as
\[
  c(X)=\half\dim X_A=\half\dim M_a=c(M).
\]
The most basic case is that of complexity zero, i.e., where all
symplectic reductions are discrete.

\begin{definition}

  A Hamiltonian $\cL_\tau(K)$-space or a \qH\ $K\tau$-manifold of
  complexity zero is called {\it multiplicity free}.

\end{definition}

Clearly, if $M$ corresponds to $X$ then one is multiplicity free if
and only if the other is. The purpose of this paper is to classify
complete multiplicity free Hamiltonian $\cL_\tau(K)$-spaces and
compact \mf\ $K\tau$-manifolds. Since both problems are equivalent, we
will investigate only \qH\ manifolds from now on.

\section{Affine root systems}\label{sec:affine}

Next, we need to recall some facts about twisted conjugacy
classes. For this and also to state the classification result, we need
to set up notation about affine root systems. Here we are following
mostly Macdonald's \cite{Mac1} and \cite{Mac2}.

Let $\Vfa$ be an Euclidean vector space, i.e., a finite dimensional
$\RR$-vector space equipped with a positive definite scalar product
$\langle\cdot,\cdot\rangle$ and let $\fa$ be an affine space for
$\Vfa$, i.e., a set with a free and transitive $\Vfa$-action
\[
  \fa\times\Vfa\to\fa:(x,t)\mapsto x+t.
\]
The set of affine linear functions on $\fa$ is denoted by $A(\fa)$. It
is an extension of the dual space $\Vfa^*$ by the constant functions
$\RR1$. The gradient of $\alpha\in A(\fa)$ is denoted by
$\Valpha\in\Vfa$. It is characterized by
\[
  \alpha(x+t)=\alpha(x)+\<\Valpha,t\>,\quad x\in\fa,t\in\Vfa
\]

Similarly, let $M(\fa)$ be the group of isometries of $\fa$
(a.k.a. motions). It is an extension of the orthogonal group
$\OG(\Vfa)$ by the group of translations $\Vfa$. More precisely, the
projection
\[
  M(\fa)\to \OG(\Vfa):w\mapsto\Vw
\]
is characterized by the property
\[
  w(x+t)=w(x)+\Vw(t),\quad x\in\fa,t\in\Vfa.
\]
For a subgroup $W$ of $M(\fa)$ let $\VW$ be its image in $\OG(\Vfa)$.

A reflection is a motion $s\in M(\fa)$ whose fixed point set is an
affine hyperplane. If $\alpha\in A(\fa)$ is a non-constant affine
linear function with zero-set $H_\alpha:=\alpha^{-1}(0)$ then
\[
  s_\alpha(x)=x-\alpha(x)\,\Valpha^\vee
\]
is the unique reflection about $H_\alpha$. Here we put, as usual,
\[
  \Valpha^\vee:=2\,\frac{\Valpha}{\norm\Valpha\norm^2}\in\Vfa.
\]
The induced action on $A(\fa)$ is then given by
\[
  s_\alpha(\beta)=\beta-\<\Vbeta,\Valpha^\vee\>\alpha\text{ for all
    $\beta\in A(\fa)$.}
\]

\begin{definition}

  An {\it Euclidean reflection group} is a subgroup
  $W\subseteq M(\fa)$ which is generated by reflections and which acts
  properly on $\fa$.

\end{definition}

Recall, that the action of $W$ is called \emph{proper} if for any
compact subset $\Omega\subseteq\fa$ there are only finitely many
elements $w\in W$ with $\Omega\cap w\Omega\ne\leer$.

A more refined notion is that of an affine root system.

\begin{definition}
  A \emph{(reduced) affine root system on $\fa$} is a subset
  $\Phi\subset A(\fa)$ having the following properties

  \begin{enumerate}
  \item $\RR1\cap\Phi=\leer$.

  \item $s_\alpha(\Phi)=\Phi$ for all $\alpha\in\Phi$.

  \item $\<\Vbeta,\Valpha^\vee\>\in\ZZ$ for all $\alpha,\beta\in\Phi$

  \item The Weyl group
    $W_\Phi:=\langle s_\alpha\mid\alpha\in\Phi\rangle\subseteq M(\fa)$
    is an Euclidean reflection group.

  \item $\RR\alpha\cap\Phi=\{\alpha,-\alpha\}$ for all
    $\alpha\in\Phi$.

  \end{enumerate}

\end{definition}

\begin{remark}

  Our definition differs from Macdonald's in two respects: First, we
  assume the root system to be reduced (last axiom). Secondly, we do
  not assume that $A(\fa)$ is spanned by $\Phi$. In fact, $\Phi$ may
  be finite or even empty. So for us all finite root systems are
  affine, as well.

\end{remark}

The classification of affine root systems is well known: First of all
there is an essentially unique orthogonal decomposition
\[
  \fa=\fa_0\times\fa_1\times\ldots\times\fa_n
\]
such that
\[
  \Phi=\Phi_1\cup\ldots\cup\Phi_n,
\]
with $\Phi_\nu\subset\fa_\nu$ irreducible for every $\nu\ge1$. So each
$\Phi_\nu$ corresponds either to a finite or to a twisted affine
Dynkin diagram.

The {\it chambers of $\Phi$ (or $W_\Phi$)} are the connected
components of the complement of the union of all reflection
hyperplanes in $\fa$. The closure of a chamber is called an {\it
  alcove}. It is known, that $W_\Phi$ acts simply transitively on the
set of alcoves, that each alcove is a fundamental domain for $W_\Phi$,
and that $W_\Phi$ is generated by the reflections about the walls
(i.e., the faces of codimension one) of any fixed alcove $\cA$. These
latter reflections are called \emph{simple} with respect to $\cA$.  If
$\Phi$ is irreducible and infinite then each alcove is a simplex. For
finite root systems, alcoves are usually called Weyl chambers and are
simplicial cones. In general, an alcove is a direct product of an
affine space, a simplicial cone and a number of simplices.

Let $\VPhi$ be the image of $\Phi$ in $\Vfa$. It is a finite, but
possibly non-reduced, root system. Its Weyl group $W_\VPhi$ is the
image $\VW_\Phi$ of $W_\Phi$ in $O(\Vfa)$.

\begin{definition}\label{def:weight}

  \begin{enumerate}

  \item Let $\Phi\subset A(\fa)$ be an affine root system. A
    \emph{weight lattice for $\Phi$} is a lattice
    $\Lambda\subseteq\Vfa$ with $\VPhi\subset\Lambda$ and
    $\VPhi^\vee\subseteq\Lambda^\vee$ where
    $\Lambda^\vee=\{t\in\Vfa\mid\<t,\Lambda\>\subseteq\ZZ\}$ is the
    dual lattice of $\Lambda$.

  \item An \emph{integral root system} is a pair $(\Phi,\Lambda)$
    where $\Phi\subset A(\fa)$ is an affine root system and
    $\Lambda\subseteq\Vfa$ is a weight lattice for $\Phi$.

  \end{enumerate}

\end{definition}

Let $(\Phi,\Lambda)$ be an integral affine root system on $\fa$. Then
$A:=\Vfa/\Lambda^\vee$ is a compact torus. Its character group
$\Xi(A)=\Hom(A,U(1))$ can be identified with $\Lambda$. More
precisely, to $\chi\in\Lambda$ corresponds the character
\[
  \tilde\chi:A\to U(1): a+\Lambda^\vee\mapsto e^{2\pi i\,\<\chi,a\>}.
\]
For every affine root $\alpha\in\Phi$ we are going to write
$\tilde\alpha:=\widetilde{\Valpha}$. Dually, every
$\eta\in\Lambda^\vee$ defines a cocharacter, namely
\[
  \tilde\eta:U(1)\to A:e^{2\pi i\,t}\mapsto t\eta +\Lambda^\vee.
\]
Again, for $\alpha\in\Phi$ we write
$\tilde\alpha^\vee:=\widetilde{\Valpha^\vee}$. Then
\[
  \tilde\chi(\tilde\alpha^\vee(u))=u^{\<\chi,\Valpha^\vee\>}\text{ for
    all }\chi\in\Lambda, \alpha\in\Phi, u\in U(1).
\]
In particular,
\[\label{e1}
  \tilde\alpha(\tilde\alpha^\vee(u))=u^2\text{ for all }\alpha\in\Phi,
  u\in U(1).
\]
The Weyl group $W_\Phi$ acts on $A$ via its quotient $\VW_\Phi$. More
precisely, for $\alpha\in\Phi$ the corresponding reflection acts as
\[\label{e2}
  s_\alpha(a)=a\cdot\tilde\alpha^\vee(\tilde\alpha(a))^{-1},\quad a\in
  A.
\]

\section{Twisted conjugacy classes}\label{sec:twisted}

The geometry of twisted conjugacy classes is very well documented in
the literature for \emph{simple} groups (to be recalled below). From
this, the case of non-simple groups can be easily deduced. For this,
{\it we assume from now on that $K$ is simply connected}. Observe that
this means, in particular, that $K$ is semisimple.

We start with a simple observation.

\begin{lemma}\label{lemma:shift}

  For $u\in K$ let $\tauq:=\Ad(u)\circ\tau\in\Aut K$, i.e.,
  ${}^\tauq k=u\otau ku^{-1}$. Let $\phi:K\to K:k\mapsto
  ku^{-1}$. Then

  \begin{enumerate}

  \item The map $\phi$ intertwines the $\tau$-twisted action with the
    $\tauq$-twisted action of $K$ on $K$.

  \item Let $(M,\omega,m)$ be a \qH\ $K\tau$-manifold and
    $\overline m:=\phi\circ m:x\mapsto m(x)u^{-1}$. Then
    $(M,\omega,\overline m)$ is a \qH\ $K\tauq$-manifold.

  \end{enumerate}

\end{lemma}

\begin{proof}

  Part {\it a)} is an easy calculation. Part {\it b)} follows from
  {\it a)} and the easily verified identities $\phi^*\chi=\chi$,
  $\phi^*\Theta_\tauq=\Theta_\tau$.
\end{proof}

\begin{remark}

  The Lemma implies that the category of \qH\ $K\tau$-manifolds
  depends only on the class of $\tau$ in $\|Out|K=\Aut K/\|Inn|K$,
  i.e., on the diagram automorphism which is induced by $\tau$. Thus
  if we wish, we may assume that $\tau$ is induced by a diagram
  automorphism. On the other hand, we don't want to do that too
  excessively since sometimes arbitrary automorphisms allow for
  greater flexibility.

\end{remark}

The following facts are well known.

\begin{theorem}\label{thm:conjugacyclasses}

  Let $K$ be a simply connected compact Lie group and $\tau$ an
  automorphism of $K$. Then there is a $\tau$-stable maximal torus
  $T\subseteq K$ and an integral affine root system
  $(\Phi_\tau,\Lambda_\tau)$ on $\fa=\ft^\tau$, the $\tau$-fixed part
  of $\Lie T$, such that the following holds:

  \begin{enumerate}

  \item Let $\pr^\tau:\ft\to\fa$ be the orthogonal projection. Then
    $\VPhi_\tau=\pr^\tau\Phi(\fk,\ft)$ and
    $\Lambda_\tau=\pr^\tau\Xi(T)$. Moreover, $\Lambda_\tau$ is also
    the weight lattice (=dual of coroot lattice) of $\VPhi_\tau$.

  \item For any alcove $\cA\subseteq\fa$ of $\Phi_\tau$ the
    composition
    \[\label{eq:AK}
      c:\cA\into\fa\overset{\exp}\Pfeil K\to K\tau/K
    \]
    is a homeomorphism.

  \item For $a\in\cA$ let $u:=\exp a\in K$. Then the twisted
    centralizer
    \[
      K(a):=K_{u}=\{k\in K\mid ku\otau k^{-1}=u\}=K^\tauq\text{ (where
        $\tauq:=\Ad(u)\circ\tau$)}
    \]
    is a connected subgroup of $K$ with maximal torus
    $S:=\exp\fa=(T^\tau)^0$. Its root datum is
    $(\VPhi_\tau(a),\Lambda_\tau)$ where
    $\Phi_\tau(a):=\{\alpha\in\Phi\mid\alpha(a)=0\}$.

  \end{enumerate}

\end{theorem}

For a proof see e.g. \cite{Segal}, \cite{Mohr}, or \cite{MW}. For the
sake of computing examples, let me indicate the construction of
$\Phi_\tau$.

The automorphism $\tau$ permutes the simple factors of $K$. Thus there
exists a $\tau$-stable decomposition $K=K_1\times\ldots\times K_s$
such that $\<\tau\>$ acts transitively on the simple factors of each
$K_i$. Let $\tau_i:=\res_{K_i}(\tau)$. Suppose that
$\Phi_{\tau_i}\subseteq\fa_i=\ft_i^{\tau_i}$ is already
constructed. Then
\[
  \Phi_\tau:=\Phi_{\tau_i}\cup\ldots\cup\Phi_{\tau_s}\subseteq
  \fa:=\fa_1\oplus\ldots\oplus\fa_s.
\]

We are now reduced to the case that $\<\tau\>$ permutes the factors of
$K$ transitively. This means that $K\cong K_0^m$ with $K_0$ simple and
there is $\tau_0\in\Aut(K_0)$ such that $\tau$ acts on $K$ as
\[
  \otau(k_1,k_2,\ldots,k_m)=(k_2,\ldots,k_m,{}^{\tau_0}k_1).
\]
The twisted action on $K_0^m$ is
\[
  (k_1g_1k_2^{-1},\ldots,k_{m-1}g_{m-1}k_m^{-1},k_mg_m{}^{\tau_0}k_1^{-1}).
\]
Therefore, $1\times K_0^{m-1}\subseteq K$ acts freely on $K$ with
quotient map
\[
  K_0^m\to K_0\tau_0:(g_1,\ldots,g_m)\mapsto g_1\ldots g_m
\]
which is equivariant with respect to the first copy of $K_0$. Let
$\Phi_{\tau_0}\subset\fa_0$ be the affine root system for
$\tau_0$. Then $\fa=\fa_0$ is embedded diagonally into $\fa_0^m$ and
$\Phi_\tau$ consists of all affine linear functions of the form
$\alpha(x)=\frac1 m\alpha_0(mx)$ with
$\alpha_0\in\Phi_{\tau_0}$. Observe though that the scalar product on
$\fa$ differs from that on $\fa_0$ by a factor of $m$.

So we may assume that $K$ is simple. Let $\ft\subseteq\fk$ be a Cartan
subalgebra, $\Phi_K\subseteq\ft^*$ the corresponding root system,
$\ft_+\subseteq\ft$ a Weyl chamber, and $\Phi^+_K\subseteq\Phi_K$ the
corresponding set of positive roots. \cref{lemma:shift} allows us to
assume that $\tau$ is induced by a graph automorphism.  Let
$\fa=\ft^\tau=\{\xi\in\ft\mid \otau\xi=\xi\}$ be the space of
$\tau$-fixed points in $\ft$. Then the set
$\VPhi_\tau:=\{\pr^\tau\alpha\mid\alpha\in\Phi(\fk,\ft)\}$ of
restricted roots is a (possibly not reduced) root system on
$\ft^\tau$. Let $\oS_\tau\subseteq\VPhi_\tau$ be the set of simple
roots with respect to the Weyl chamber $\ft^\tau\cap\ft^+$.

Let $r\in\{1,2,3\}$ be the order of $\tau$. Then we define
$\theta\in\VPhi_\tau$ to be the longest dominant root if $r=1$ or
$r=2$ and $K\cong SU(2n+1)$ (case $\sA_{2n}^{(2)}$). Otherwise,
$\theta$ denotes the dominant short root of $\VPhi_\tau$. Define the
affine linear function $\alpha_0(x)=-\<\theta,x\>+\frac{2\pi}r$. Then
$\Phi_\tau$ is the affine root system whose set of simple roots is
\[
  S_\tau:=\oS_\tau\cup\{\alpha_0\}.
\]

A first application of the description of twisted conjugacy classes
goes as follows. Let $m:M\to K\tau$ be a moment map and
$\cA\subseteq\fa$ an alcove. Since $c:\cA\to K\tau/K$ (see
\eqref{eq:AK}) is bijective one can define the \emph{invariant moment
  map} as
\[
  m_+:=c^{-1}\circ m:M\to\cA.
\]
This yields a commutative diagram
\[
  \cxymatrix{M\ar[r]^m\ar[d]_{m_+}&K\ar@{>>}[d]\\\cA\ar[r]^<<<c_<<<{\rm
      bij}&K\tau/K}
\]
Observe that $c$ is in general just continuous but not smooth, so the
same holds true for $m_+$.

\begin{definition}

  Let $m:M\to K\tau$ be \qH. Then $\cP_M=m_+(M)\subseteq\cA$ is called
  the {\it momentum image} of $M$.

\end{definition}

Observe that $\cP_M$ determines the actual image of $m$ since
$m(M)=K\cdot\exp\cP_M$. Fundamental is the following

\begin{theorem}[\cite{AMM}*{Thm.~7.2}, \cite{Meinrenken}*{Thm.~4.4}]

  Let $K$ be simply connected and let $(M,m)$ be a connected compact
  \qH\ $K\tau$-manifold with moment map $m:M\to K\tau$. Then its
  momentum image $\cP_M$ is a convex polytope lying inside
  $\cA$. Moreover, all fibers of $m$ (and therefore $m_+$) are
  connected.
 
\end{theorem}

Since we want to glue compact multiplicity free manifolds from local
pieces we have to weaken the compactness property.

\begin{definition}

  A \mf\ $K$-manifold $M$ is called {\it convex} if its momentum image
  $\cP_M$ is convex and locally closed in $\cA$.

\end{definition}

For example, if $M$ is compact and multiplicity free and
$U\subseteq\cA$ is any convex open subset then $M_U:=m_+^{-1}(U)$ is
multiplicity free and convex in the above sense.

\section{The local structure of quasi-Hamiltonian manifolds}
\label{sec:localstructure}

The local structure of the space of twisted conjugacy classes is also
well known from, e.g., \cites{Mohr,MW}. Let $a\in\cA$ (notation as in
\cref{thm:conjugacyclasses}) and $u:=\exp(a)\in K$. Let
\[
  L:=\{l\in K\mid lu\otau l^{-1}=u\}=\{l\in K\mid \otau l=u^{-1}lu\}
\]
be its (twisted) stabilizer in $K$. Observe that $\otau u=u$ implies
$u\in L$ and $\otau L=L$. Then an easy calculation shows that
\[
  \phi:L\to K: l\mapsto lu
\]
is an $L$-equivariant map with $\phi(e)=u$ where $L$ acts on itself
and $K$ by untwisted and $\tau$-twisted conjugation, respectively. Let
$\tauq:=\Ad(u)\circ\tau$ and let $\cO\subseteq K$ be the twisted
conjugacy class of $u$. Then another easy calculation gives for its
tangent space
\[
  T_u\cO=\{(\xi-{}^\tauq\xi)u\mid\xi\in\fk\}={}^{1-\tauq}\fk\ u
\]
On the other hand
\[
  \Lie L=\ker(1-\tauq)\subseteq\fk.
\]
This shows that $\phi(L)$ is a slice for $\cO$ in $u$. Now let
$U\subseteq\cA$ be an open neighborhood of $a$ which is small enough
such that $U_0:=U-a$ is open in the cone
$\cC:=\RR_{\ge0}(\cA-a)$. Observe that $\cC$ is a Weyl chamber of $L$
by \cref{thm:conjugacyclasses}. Then $L_U:=\Ad L\,U_0$ is an open
neighborhood of $1\in L$. This determines the local structure of
$K\tau$ near $\cO$:

\begin{lemma}\label{lemma:slice}

  The map
  \[
    \Phi: K\times^LL_U\to K\tau:[k,l]\mapsto k\phi(l)\otau
    k^{-1}=klu\otau k^{-1}
  \]
  is a $K$-equivariant diffeomorphism onto an open neighborhood of
  $\cO$ in $K$.

\end{lemma}

Now let $m:M\to K\tau$ be a \qH\ manifold. Then the pull-back with
$\phi$ yields the $L$-manifold $M_U=M\times_KL_U$ such that the
following diagram commutes
\[\label{eq:MU}
  \cxymatrix{M_U\ar@{^(->}[r]\ar[d]^{m_L}&M\ar[d]^m\\L\ar@{^(->}[r]^\phi&K\\}
\]
(where $m_L$ has in fact values in $L_U$). Then in \cite{AMM} it was
shown that $M_U$ carries canonically the structure of a
quasi-Hamiltonian $L$-manifold with moment map $m_L$ and $2$-form
$\omega_L=\omega|_{M_U}$. More generally, the following holds:

\begin{proposition}\label{prop:local-qH}
  Every $a\in\cA$ has an open neighborhood $U\subseteq \cA$ such that
  the functor $M\mapsto M_U$ is an equivalence between the category of
  quasi-Hamiltonian $K\tau$-manifolds $(M,m)$ with $m_+(M)\subseteq U$
  and (untwisted) quasi-Hamiltonian $L$-manifolds $(M',m')$ with
  $m'_+(M')\subseteq U-a$.

\end{proposition}

\begin{proof}

  \cref{lemma:shift} with $u=\exp a$ allows to replace $\tau$ by
  $\tauq=\Ad(u)\circ\tau$. Thereby, we may assume that $a=0$. We
  proceed to construct a functor which is quasi-inverse to
  $M\mapsto M_U$. For this, recall from \cite{AMM} the \emph{double
    $D(K)$ be of $K$}. It is a \qH\ $K\times K$-manifold which equals
  $K\times K$ as a manifold. The action of $K\times K$ on $D(K)$ is
  given by
  \[
    (k_1,k_2)*(b_1,b_2)=(k_1b_1k_2^{-1},k_2b_2k_1^{-1}).
  \]
  The moment map is
  \[
    m_{D(K)}:D(K)\to K\times K:(b_1,b_2)\mapsto
    (b_1b_2,b_1^{-1}b_2^{-1}).
  \]
  The \emph{twisted double} is the open subset
  $D_\tau(K)=K\times K\tau$ of $D(\ZZ\tau\ltimes K)$. By identifying
  $D_\tau(K)$ with $K^2$ we get the twisted action
  \[
    (k_1,k_2)*(b_1,b_2)=(k_1b_1k_2^{-1},k_2b_2\otau k_1^{-1}).
  \]
  The moment map takes values in $K\tau\times \tau^{-1}K$. After
  identifying also that space with $K^2$ we get the
  $\tau\times\tau^{-1}$-twisted moment map
  \[
    m_{D_\tau(K)}:D_\tau(K)\to K\times K:(b_1,b_2)\mapsto(b_1b_2,\otau
    b_1^{-1}b_2^{-1}).
  \]

  Next put $L_0:=\Ad L(\exp(U-a))$. It follows from
  \cref{thm:conjugacyclasses} that $L_0$ is is a conjugation invariant
  open neighborhood of $1\in L$. Now put
  \[
    Z:=D_\tau(K)_{K\times L_0^{-1}}=\{(b_1,b_2)\in K\times K\mid
    b_2\otau b_1\in L_0\}.
  \]
  This is a \qH\ $K\tau\times L$-manifold which we can identify with
  $K\times L_0$ via the map
  \[
    K\times L_0\overset\sim \to Z:(b,c)\mapsto (b,c\otau b^{-1}).
  \]
  Because of $\otau l=l$ for all $l\in L$, the induced
  $K\times L$-action on $K\times L_0$ is
  \[
    (k,l)*(b,c)=(kbl^{-1},lcl^{-1}).
  \]
  while the moment map becomes
  \[
    m_Z:K\times L_0\to K\times L:(b,c)\mapsto (bc\otau b^{-1},c^{-1}).
  \]

  In \cite{AMM} also the \emph{fusion} of two \qH\ manifolds was
  introduced. More precisely, let $M_1$ be a \qH\ $K\times L$-manifold
  with moment map $(m_1,m_2)$ and let $M_2$ be a \qH\ $L$-manifold
  with moment map $m_3$. Then it was shown that $M_1\times M_2$ is a
  \qH\ $K\times L$-manifold with action
  $(k,l)(x_1,x_2)=(kx_1l^{-1},lx_2)$ and moment map
  $(x_1,x_2)\mapsto(m_1(x_1),m_2(x_1)m_3(x_2))$. We are going to
  denote this new manifold by $M_1\otimes_L M_2$. This construction
  also works if the action of $K$ is twisted.

  Finally, recall from \cite{AMM} also the process of \emph{symplectic
    reduction} of a \qH\ $K\times L$ manifold $M$ with moment map
  $(m_1,m_2)$. Assume that $1\in L$ is a regular value for $m_2$. Then
  it is shown that
  \[
    M\mod L:=m_2^{-1}(1)/L
  \]
  is a \qH\ $K$-manifold with moment map $Lx\mapsto m_1(x)$. Also this
  works for a twisted $K$-action.

  Now the desired quasi-inverse functor is
  \[
    \|ind|_L^KM':=(Z\otimes_L M')\mod L
  \]
  which is \qH\ by construction. Moreover, by definition
  \[
    \|ind|_L^KM'=\{(b,c,x)\in K\times L_0\times M'\mid
    c^{-1}m'(x)=1\}/L\cong (K\times M')/L=K\times^LM'.
  \]
  An easy calculation shows that the moment map induces on
  $K\times^LM'$ the map
  \[
    K\times^KM'\to K:[b,x]\mapsto bm'(x)\otau b^{-1}.
  \]

  Thus, it suffices to show that the natural maps
  \[
    \phi:M'\to K\times^LM':x\mapsto[1,x]
  \]
  and
  \[
    \psi:K\times^LM_U\to M:[k,x]\mapsto kx
  \]
  are isomorphisms of quasi-Hamiltonian manifolds. First, it follows
  immediately from \cref{lemma:slice} that both $\phi$ and $\psi$ are
  $K$-equivariant diffeomorphisms which are compatible with the moment
  maps. It remains to show that the 2-forms match up.

  For $\phi$ observe that $x\in M'$ is mapped to the class of
  $(1,m(x),x)\in K\times L_0\times M'$.  Now recall the explicit
  formula of \cite{AMM}*{Thm.~6.1} for the $2$-form on the fusion
  product of two \qH\ manifolds $(M_1,\omega_1,m_1)$ and
  $(M_2,\omega_2,m_2)$:
  \[\label{eq:fusion}
    \pi_1^*\omega_1+\pi_2^*\omega_2+\frac12\<m_1^*\theta,m_2^*\thetaq\>.
  \]
  which we apply to $M_1=Z$ and $M_2=M'$.  Let
  $\iota:L\to L:h\mapsto h^{-1}$ be the inversion. Because of
  $\iota^*\theta=-\thetaq$ the pull-back of the third term in
  \eqref{eq:fusion} to $M'$ vanishes. To see that also the first
  summand vanishes on $M'$ we look at the explicit form of $\omega_1$
  on $D(K)$ (see \cite{AMM}*{Prop.~3.2}):
  \[
    \omega_1=\half\<p_1^*\theta,p_2^*\thetaq\>+\half\<p_1^*\thetaq,p_2^*\theta\>
  \]
  where $p_1,p_2$ are the two projections of $D(K)$ to $K$. The map
  from $M'$ to $Z\subseteq D(K)$ is $x\mapsto(1,m'(x))$. Hence $p_1$
  is constant on $M'$ implying that the pull-backs of $p_1^*\theta$
  and $p_1^*\thetaq$, hence of $\omega_1$ to $M'$ vanish. This
  finishes the proof that $\phi$ is an isomorphism of
  quasi-Hamiltonian manifolds.

  To show this for $\psi$ let $\omega$ be the given $2$-form on
  $M$. We claim that $\omega$ is uniquely determined by the moment map
  property of $m$ and its restriction to $M_U$. By $K$-invariance,
  $\omega$ is determined by its values in any $x\in M_U$. The moment
  map property \cref{D1}\ref{D1i3} allows to compute
  $\omega(\xi,\eta)$ where $\xi\in\fk x$ and $\eta\in T_xM$. Moreover,
  also $\omega(\xi,\eta)$ is known for $\xi,\eta\in T_xM_U$. Because
  of $\fk_x+T_xM_U=T_xM$ we proved our claim. Now let $\omega'$ be the
  2-form on $K\times^LM_U$. Then the claim shows
  $(\psi^{-1})^*\omega'=\omega$ which is what had to be proved.
\end{proof}

Now we pull everything back to the Lie algebra $\fl$ of $L$ using the
exponential map. For this let $\fl_0:=\Ad L(U-a)\subseteq\fl$. It is
an $L$-invariant open neighborhood of $0\in\fl$ such that
$\exp:\fl_0\to L_0$ is a diffeomorphism. We denote the inverse of this
diffeomorphism by $\log_U$.

Now assume that $(M,\omega,m_0)$ is a Hamiltonian manifold in the
ordinary sense. This means in particular that $m_0$ is an
$L$-equivariant map from $M$ to $\fl^*$ which we continue to identify
with $\fl$. The cone spanned by $U-a$ in $\ft$ is a Weyl chamber
$\ft^+$ for $\fl$. Hence we get a homeomorphism $\ft^+\to \fl\mod
L$. Inverting it, one can also define the invariant moment map
\[
  (m_0)_+:M\to \fl\to \fl\mod L \overset\sim\to\ft^+.
\]

To get a Hamiltonian manifold from a \qH\ one, one defines the
$2$-form $\tilde\omega$ on $\fl$ by
\[
  \tilde\omega_\lambda(\xi_1,\xi_2)=\<g(\Ad\lambda)\xi_1,\xi_2\>
\]
where $\lambda,\xi_1,\xi_2\in\fl$ and
\[
  g(x):=\frac{\sinh x
    -x}{x^2}=\frac{x}{3!}+\frac{x^3}{5!}+\frac{x^5}{7!}+\ldots
\]
An easy calculation shows that this two-form equals the two form
$\varpi$ in Lemma~3.3 of \cite{AMM}. Now for a quasi-Hamiltonian
$L$-manifold $(M,\omega,m)$ with $m(M)\subseteq U$ we put
$m_0:=\log_U\circ m:M\to\fl\cong\fl^*$ and
$\omega_0:=\omega-m_0^*\tilde\omega$.

\begin{lemma}

  The functor $(M,\omega,m)\mapsto\log M:=(M,\omega_0,m_0)$ is an
  equivalence between the category of quasi-Hamiltonian $L$-manifolds
  $(M,\omega,m)$ with $m_+(M)\subseteq U$ and Hamiltonian
  $L$-manifolds $(M_0,m_0)$ with $(m_0)_+(M_0)\subseteq U$.

\end{lemma}

\begin{proof}

  This is \cite{AMM}*{Prop.~3.4 and Rem.~3.3}. The quasi-inverse
  functor is
  \begin{equation*}
    (M_0,\omega_0,m_0)\mapsto\exp
    M_0:=(M_0,\omega_0+m_0^*\tilde\omega,\exp\circ
    m_0)\qedhere
  \end{equation*}\end{proof}

Putting both constructions together we get:

\begin{theorem}\label{thm:cross}

  Any $a\in\cA$ has an open neighborhood $U\subseteq\cA$ such that
  there is an equivalence between the category of quasi-Hamiltonian
  $K\tau$-manifolds $(M,m)$ with $m_+(M)\subseteq U$ and Hamiltonian
  $L$-manifolds $(M_0,m_0)$ with $(m_0)_+(M_0)\subseteq U-a$.

\end{theorem}

\begin{remarks}

  \textit{i)} The theorem hold only if the manifolds are allowed to be
  non-connected.  The reason for this is that even if $M$ is
  connected, the local model $M_0$ might be not.

  \textit{ii)} It can be shown that the only requirement for the open
  subset $U$ of $\cA$ is the validity of the Slice
  \cref{lemma:slice}. As already remarked in \cite{AMM}*{Rem.\ 7.1}
  this implies that for any $L$ there is a canonical open set
  $U$. More precisely, let $\cA^\sigma\subseteq\cA$ be a face of
  $\cA$. Then $\cA_\sigma\subseteq\cA$ is obtained by removing all
  faces which do not contain $\cA^\sigma$:
  \[\label{eq:openstar}
    \cA_\sigma:=\cA\setminus\bigcup_{\cA^\sigma\not\subseteq\cA^\eta}\cA^\eta
  \]
  This is an open subset of $\cA$. Now choose
  $a\in\cA^\sigma\cap\cA_\sigma$. Then the stabilizer $K_\sigma=K_a$
  is independent of the choice of $a$. Then \cref{thm:cross} holds for
  $L=K_\sigma$ and $U=\cA_\sigma$.

\end{remarks}

Later, we need the property that the equivalence of categories above
is compatible with Hamiltonian dynamics. For this, let $f$ be a
$K$-invariant smooth function on $M$. Then it was shown in
\cite{AMM}*{Prop.~4.6} that $M$ carries a unique vector field $H_f$
with
\[\label{eq:q-HamVF}
  \iota(H_f)\omega=df\text{ and }\iota(H_f)m^*\theta=0.
\]
Observe, that he second condition is only necessary when $\omega$ is
degenerate.  Now let $f_0:=f|_{M_U}$ which can be considered as an
$L$-invariant function on $M_0=M_U$. It induces a Hamiltonian vector
field in the classical sense:
\[\label{eq:HamVF}
  \iota(H_{f_0})\omega_0=df_0.
\]

\begin{lemma}\label{qHamFlow}

  Let $f$ be a $K$-invariant smooth function on $M$ and $f_0$ its
  restriction to $M_U$. Then $H_{f_0}$ coincides with $H_f$ on
  $M_U$. In particular, the Hamiltonian flow generated by $f$ on $M$
  preserves $M_U$ and coincides there with the one generated by $f_0$
  on $M_0$.

\end{lemma}

\begin{proof}

  The vector field $H_f$ is parallel to the fibers of $m$. Since
  $m=\exp\circ m_0$, it is on $M_0$ also parallel to the fibers of
  $m_0$. Hence it lies in the kernel of $m_0^*\tilde\omega$. This
  implies
  \[
    \iota(H_f)\omega_0=\iota(H_f)\omega|_{M_U}=df|_{M_U}=df_0
  \]
  which is the characterizing equation of $H_{f_0}$.
\end{proof}

\section{Classification of \mf\ manifolds}\label{sec:class}

In this section we describe our main classification result in
detail. The proof will be given in the subsequent sections.

In general, the moment map can have quite pathological properties. See
e.g. \cite{KnopConvexity}. If the manifold is compact these
pathologies won't occur because of the following

\begin{theorem}

  Assume $K$ to be simply connected and let $M$ be a compact \qH\
  manifold. Then $m_+(M)$ is a convex subset of $\cA$ and
  $m_+/K:M/K\to m_+(M)$ is a homeomorphism.

\end{theorem}

\begin{proof}

  By \cite{AMM}*{Thm.~7.2}, the image $m_+(M)$ is convex and the
  fibers of $m$ and therefor of $m_+/K$ are connected. On the other
  hand, these fibers are discrete since $M$ is multiplicity
  free. Hence $m_+/K$ is bijective and hence, by compactness, a
  homeomorphism.
\end{proof}

The conclusion of the theorem is inherited by $M_U$ whenever the open
set $U$ itself is convex. Since we want to glue $M$ from local pieces,
we define:

\begin{definition}

  A \qH\ manifold $(M,m)$ is called \emph{convex and multiplicity free
  } if

  \begin{enumerate}

  \item The \emph{momentum image} $\cP_M:=m_+(M)$ is a convex subset
    of $\cA$.

  \item\label{it:convex2} The map $m_+/K:M/K\to\cP_M$ is a
    homeomorphism.

  \end{enumerate}

\end{definition}

Now let $a\in\cP_M$ and $U\subseteq\cA$ such that the conclusion of
the local structure theorem holds. Without loss of generality we may
assume that $U$ is convex. Then $M_0$ is a convex multiplicity free
Hamiltonian $L$-manifold. For its momentum image holds
\[
  \cP_{M_0}=\cP_M\cap U.
\]
Observe that \ref{it:convex2} implies that the fibers of $m_+$ are
$K$-orbits, hence connected. Therefore $M_0$ is connected, as
well. Recall the space $\fa$ from \cref{thm:conjugacyclasses}. Then we
define the \emph{tangent cone} of a subset $\cP\subseteq\cA$ in
$a\in\cP$ as
\[
  C_a\cP:=\RR_{\ge0}(\cP-a)\subseteq\fa.
\]
The discussion above and \cite{KnopConvexity}*{Thm.~2.7} imply the
following structural property of $\cP_M$.

\begin{lemma}

  Let $M$ be a convex, \mf\ $K\tau$-manifold. Then $\cP_M$ is a
  locally polyhedral set, i.e., for every $a\in\cP_M$ the tangent cone
  $C_a:=C_a\cP_M$ is a finitely generated convex cone and there is an
  open neighborhood $U$ of $a$ in $\cA$ with
  \[
    \cP_M\cap U=C_a\cap U.
  \]

\end{lemma}

From this we get

\begin{corollary}

  Let $\fa_M\subseteq\fa$ be the affine subspace spanned by
  $\cP_M$. Then the interior $\cP_M^0$ of $\cP_M$ in $\fa_M$ is
  non-empty and dense in $\cP_M$.

\end{corollary}

The dimension of $\fa_M$ is an important invariant of $M$, called the
\emph{rank} $\rk M$.

Next, we study the generic isotropy group of $M$. Lemma~2.3 of
\cite{KnopAuto} combined with \cref{thm:cross} (reduction to the local
case) implies

\begin{lemma}

  Let $a\in\cP_M^0$ and let $L_0\subseteq L$ be the kernel of the
  $L$-action on $M_0$. Then $\overline A_M:=L/L_0$ is a torus and
  $L_0$ is a principal isotropy group for $K$ on $M$.

\end{lemma}

The group $L_0$ can be encoded by a lattice as follows. Since
$\fl=\fa+\fl_0$, the orthogonal complement $\Vfa_M$ of $\fa\cap\fl_0$
in $\fa$ can be identified with $\Lie A_M$. It follows from the
properties of a moment map that $\Vfa_M$ is the group of translations
of the affine space $\fa_M$. Indeed let
$\tilde\fa_M:=\fa_M-\fa_M$. Then
\[
  \xi\in\tilde\fa_M^\perp\Leftrightarrow \<\xi,\cdot\>\text{ is
    constant on $\fa_M$} \Leftrightarrow \xi_*=0\Leftrightarrow
  \xi\in\fa\cap\fl_0\Leftrightarrow\xi\in\Vfa_M^\perp.
\]
This shows that the Lie algebra $\fl_0$ is already determined by the
momentum image $\cP_M$. For the group itself, one needs additionally
the lattice
\[
  \Lambda_M^\vee:=\ker\exp:\Vfa_M\to A_M.
\]
In the following, we prefer to work with its dual lattice
\[
  \Lambda_M:=\{x\in\Vfa_M\mid\<x,\Lambda_M^\vee\>\subseteq\ZZ\}\subseteq\Vfa_M
\]
which can be also interpreted as the character group $\Xi(A_M)$ of
$A_M$.  We call $\Lambda_M$ the \emph{character group} of $M$.

We are going to classify multiplicity free manifolds in terms of the
pair $(\cP_M,\Lambda_M)$. To describe which pairs are possible we need
some notions from algebraic geometry.

Let $G=K_\CC$ be the complexification of $K$. This is a connected
complex reductive group. Let $B\subseteq G$ be a Borel subgroup and
let $\cX(B):=Hom(B,\CC^*)\cong\ZZ^{\operatorname{rk}G}$ be its
character group. It is possible to identify $\cX(B)\otimes\RR$ with a
Cartan subalgebra $\ft$ of $\fk$ (actually the dual of one). The
characters which lie in the Weyl chamber $\ft^+$ are called
dominant. Recall, that that there is a $1:1$-correspondence
$\chi\mapsto L(\chi)$ between dominant characters and irreducible
representations of $G$.

A $G$-variety $X$ is called \emph{spherical} if $B$ has an open dense
orbit in $X$. In the following we are only interested in affine
varieties. In this case, there is a purely representation theoretic
criterion for sphericality due to Vinberg-Kimel{\cprime}fel{\cprime}d,
\cite{VK}: Let $\CC[X]$ be the ring of regular functions on $X$. Then
$\CC[X]$ is in particular a representation of $G$ and decomposes as a
direct sum of irreducible representation. The criterion states that
$X$ is spherical if and only if $\CC[X]$ is a multiplicity free module
for $G$, i.e., no irreducible representation appears in $\CC[X]$ more
than once. Under these conditions there is a subset
$\Lambda_X^+\subseteq\cX(B)\cap\ft^+$ such that
\[
  \CC[X]=\bigoplus_{\chi\in\Lambda_X^+}L(\chi)
\]
as a $G$-representation. The set $\Lambda_X^+$ is actually additively
closed and is called the \emph{weight monoid of $X$}.

Now we return to the compact group $K$. Let $\cP\subseteq\cA$ be a
locally closed convex subset. Let $\fa_\cP\subseteq\ft$ be the affine
space spanned by $\cP$ and let
\[
  \Vfa_\cP:=\fa_\cP-\fa_\cP
\]
be its group of translations. A point $a\in\cP$ gives also rise to an
element $u=\exp(a)\in K\tau$ and hence to an isotropy group $K(a)$
with respect to the twisted action. Observe that the tangent cone of
$\cA$ in $a$ is a Weyl chamber for $K(a)$.

\begin{definition}
  Let $K$ be a simply connected compact Lie group with automorphism
  $\tau$ and fundamental alcove $\cA$. Let $\cP\subseteq\cA$ be a
  locally closed convex subset and $\Lambda\subseteq\Vfa_\cP$ a
  lattice. Then $(\cP,\Lambda)$ is called \emph{spherical in
    $a\in\cP$} if
  \begin{enumerate}

  \item $\cP$ is polyhedral in $a$, i.e.,
    \[
      \cP\cap U=(a+C_a\cP)\cap U
    \]
    for a neighborhood $U$ of $a$ in $\cA$, and

  \item there is a smooth affine spherical $K(a)_\CC$-variety $X$ such
    that
    \[
      C_a\cP\cap\Lambda=\Lambda_X^+.
    \]
  \end{enumerate}
  The pair $(\cP,\Lambda)$ is \emph{spherical} if it is spherical in
  all $a\in\cP$.

\end{definition}

Here is our main result:

\begin{theorem}\label{thm:main}

  Let $K$ be a simply connected compact Lie group with twist
  $\tau$. Then the map $M\mapsto(\cP_M,\Lambda_M)$ furnishes a
  bijection between
  \begin{itemize}

  \item isomorphism classes of convex \mf\ $K\tau$-manifolds and

  \item spherical pairs $(\cP,\Lambda_M)$.

  \end{itemize}

  Under this correspondence, $M$ is compact if and only if $\cP_M$ is
  closed in $\cA$.

\end{theorem}

In the remainder of this section we reduce the proof of the main
\cref{thm:main} to a statement about automorphisms. We start with:

\begin{lemma}

  Let $M$ be a convex \mf\ $K\tau$-manifold. Then the pair
  $(\cP_M,\Lambda_M)$ is spherical.

\end{lemma}

\begin{proof}

  Using the cross section theorem \ref{thm:cross}, the problem is
  reduced to the Hamiltonian case. Then it is part of
  \cite{KnopAuto}*{Thm.~11.2}.
\end{proof}

Next, we state local existence and uniqueness:

\begin{lemma}\label{lemma:local}

  Let $(\cP,\Lambda)$ be a spherical pair.

  \begin{enumerate}

  \item\label{it:local1} For every point $a\in\cP$ there is a convex
    \mf\ $K\tau$-manifold $M$ with $\Lambda_M=\Lambda$ and such that
    $\cP_M$ is an open neighborhood of $a$ in $\cP$.

  \item\label{it:local2} Let $(M,m)$, $(M',m')$ be two convex \mf\
    $K\tau$-manifolds with
    $(\cP_M,\Lambda_M)=(\cP_{M'},\Lambda_{M'})=(\cP,\Lambda)$. Then
    there is an open cover $\cP=\bigcup_\nu\cP_\nu$ and for all $\nu$
    isomorphisms of \qH\ $K\tau$-manifolds
    \[
      m_+^{-1}(\cP_\nu)\cong (m')_+^{-1}(\cP_\nu).
    \]

  \end{enumerate}

\end{lemma}

\begin{proof}

  Again, by localizing, we may assume that $M$ is Hamiltonian. Then
  \ref{it:local1} holds by the definition of a spherical pair and
  \ref{it:local2} is a basically a result of Losev \cite{Losev} (see
  also \cite{KnopAuto}*{Thm.~2.4}).
\end{proof}

Next, we consider automorphisms.

\begin{lemma}

  Let $M$ be a convex, multiplicity free Hamiltonian or \qH\
  manifold. Then its automorphism group is abelian.

\end{lemma}

\begin{proof}

  By localization, one can assume that $M$ is Hamiltonian. Then use
  \cite{KnopAuto}*{Thm.~9.2}.
\end{proof}

This can be used as follows: Let $(\cP,\Lambda)$ be a spherical pair
and let $\cP_0\subseteq\cP$ be open such that there is a convex
multiplicity free manifold $M$ with
$(\cP_M,\Lambda_M)=(\cP_0,\Lambda)$. Then
\[
  \fL_{\cP,\Lambda}(\cP_0):=\Aut M
\]
depends only on $\cP_0$ and not on the choice of $M$. This follows
from the fact that any isomorphism $M\cong M'$ induces an isomorphism
of the automorphism groups which is unique up to conjugation, so
canonical because of abelianness.

Since $\fL_{\cP,\Lambda}$ is a sheaf of abelian groups it has
cohomology groups. Now, the technical heart of this paper is:

\begin{theorem}

  Let $(\cP,\Lambda)$ be a spherical pair. Then
  $H^i(\cP,\fL_{\cP,\Lambda})=0\text{ for all $i\ge1$}$.

\end{theorem}

The proof will be given in the next sections. See
\cref{thm:globalroots} for a description of $\fL_{\cP,\Lambda}$ and
\cref{T1} for the cohomology vanishing.

Now we can prove the main \cref{thm:main}: the vanishing of $H^2$
implies that the local models from \cref{lemma:local}\ref{it:local1}
glue to a global model $M$. The vanishing of $H^1$ implies that all
local isomorphisms from \cref{lemma:local}\ref{it:local2} glue to a
global isomorphism.

\begin{remark}

  In modern parlance, the argument goes as follows: Let
  $(\cP,\Lambda)$ be a spherical pair and let $\MF_{\cP,\Lambda}$ be
  the category of (locally) convex \mf\ manifolds $M$ with
  $\Lambda_M=\Lambda$ and $\cP_M\subseteq\cP$ open. Then
  \cref{lemma:local} means that $\MF_{\cP,\Lambda}$ is a \emph{gerbe}
  over $\cP$. Since all automorphism groups are abelian, its
  \emph{band} $\fL_{\cP,\Lambda}$ is a sheaf of abelian groups. Now
  $H^2(\cP,\fL_{\cP,\Lambda})=0$ means that $\MF_{\cP,\Lambda}$ is
  equivalent to the category $\|Tors|\fL_{\cP,\Lambda}$ of
  $\fL_{\cP,\Lambda}$-torsors. Because of the vanishing of $H^1$, all
  torsors over $\cP$ are trivial. So $\MF_{\cP,\Lambda}$ contains up
  to isomorphism exactly one object $M$ with $\cP_M=\cP$.

\end{remark}

\section{The automorphism group of a multiplicity free manifold}
\label{sec:auto}

Return to the notation of section \ref{sec:affine}, i.e., $\Vfa$ is an
Euclidean vector space, $\fa$ is an affine space for $\Vfa$, and
$\Phi$ is a root system on $\fa$ with fundamental alcove
$\cA\subseteq\fa$. Let moreover $\Lambda\subseteq\Vfa$ be a weight
lattice for $\Phi$. This defines the torus $A:=\Vfa/\Lambda^\vee$.

Let $\cP\subseteq\cA$ be a locally polyhedral convex subset with
non-empty interior $\cP^0$. This means, in particular, that
$\cP^\circ$ is dense in $\cP$. We proceed to define a number of
properties that a map $\phi:\cP\to A$ may have.

\begin{enumerate}

\item A map $\phi:\cP\rightarrow A$ is called \emph{smooth} if for
  every $x\in\cP$ there is a smooth map $\tilde\phi:U\rightarrow A$
  where $U$ is an open neighborhood of $x\in\fa$ and
  $\tilde\phi|_{\cP\cap U}=\phi|_{\cP\cap U}$. Let $\hat\cC_{\fa,x}$
  and $\hat\cC_{A,a}$ be the completions of the local ring of smooth
  functions (i.e., formal power series) in $x\in\fa$ and $a\in A$,
  respectively. Then if $\phi$ is smooth with $\phi(x)=a$ it induces
  an algebra homomorphism
  $\hat\phi_x:\hat\cC_{A,\phi(x)}\to\hat\cC_{\fa,x}$. In fact, by
  continuity, $\hat\phi_x=\hat{\tilde\phi}_x$ is independent of the
  choice of $\tilde\phi$ since $\cP^0$ is dense in $\cP$.

\item The scalar product on $\Vfa$ induces a canonical symplectic
  structure on the product spaces of $\fa\times A$ by
  \[
    \omega(\xi_1+\eta_1,\xi_2+\eta_2)=\<\xi_1,\eta_2\>-\<\xi_2,\eta_1\>.
  \]
  After the identifications $\fa\cong\Vfa\cong\Vfa^*$ this is just the
  canonical symplectic form on the cotangent bundle $T^*_A$.  Now, a
  smooth map $\phi:\cP\to A$ is called \emph{closed} if the graph of
  $\phi|_{\cP^0}$ is a Lagrangian submanifold of $\fa\times A$.

\item Recall that the Weyl group $W$ of $\Phi$ fixes the lattice
  $\Lambda$ and therefore acts on $A$ by way of group
  automorphisms. Let $W_0$ be a subgroup of $W$ and for any $x\in\cP$
  let $(W_0)_x$ be its isotropy group. Then a smooth map
  $\phi:\cP\to A$ is called \emph{$W_0$-equivariant} if for every
  $x\in\cP$ the point $a=\phi(x)$ is a $(W_0)_x$-fixed point and the
  induced homomorphism
  \[
    \hat\phi_x:\hat\cC_{A,a}\to\hat\cC_{\fa,x}
  \]
  is $(W_0)_x$-equivariant.

\item Let $\Phi_0\subseteq\Phi$ be a subroot system and let $W_0$ be
  its Weyl group.  A smooth map $\phi:\cP\rightarrow A$ is called {\it
    $\Phi_0$-equivariant} if it is $W_0$-equivariant and
  \[\label{eq:atpx}
    \tilde\alpha(\phi(x))=1
  \]
  for all $x\in \cP$ and all roots $\alpha\in\Phi_0$ with
  $\alpha(x)=0$.

\end{enumerate}

We comment on these notions a bit more.

First of all, the notion of closedness can be rephrased in two
ways. Since $\exp:\Vfa\to A$ is a covering and $\cP$ is simply
connected, the mapping $\phi$ can be lifted to a smooth map
$\tilde\phi:\cP\to\Vfa$. Because of the identification
$\Vfa\cong\Vfa^*$ one can think of $\tilde\phi$ as a $1$-form. Then it
easy to see that $\phi$ is closed if and only if $\tilde\phi$ is a
closed $1$-form (whence the name). Moreover, this shows that all
closed maps are of the form $\exp(\nabla f)$ where $f$ is a smooth
function on $\cP$ and $\nabla f$ is its gradient.

For another way to see closedness, consider the derivative of $\phi$
at $x$ which is an endomorphism
\[
  D_x\phi:\Vfa=T_x\fa\to T_{\phi(x)}A=\Vfa.
\]
Then $\phi$ is closed if and only if $D_x\phi$ is a self-adjoint
operator for all $x\in\cP$

Since $\cP$ lies in $\cA$, the fundamental alcove of $W$, every
$W_0$-orbit of $\fa$ meets $\cP$ in at most one point. Therefore it
can be shown (non-trivially!) that the condition of $W_0$-equivariance
is equivalent to the existence of a smooth (honestly)
$W_0$-equivariant map $\tilde\phi:W_0\cP\to A$.

Finally, what is the difference between $W_0$- and
$\Phi_0$-equivariance? Actually not much. For $\alpha\in\Phi_0$ let
$s_\alpha\in W_0$ be the corresponding reflection. Then
$s_\alpha\in(W_0)_x$ if and only if $\alpha(x)=0$. In this case,
$W_0$-equivariance implies $s_\alpha(a)=a$ where $a=\phi(x)$. This
means
\[\label{eq:1A}
  \tilde\alpha^\vee(\tilde\alpha(a))=1_A
\]
by equation \eqref{e2}. Applying $\tilde\alpha$ to both sides, we see
(equation \eqref{e1}) that
\[
  \tilde\alpha(\phi(x))^2=1
\]
follows already from $W_0$-equivariance. So $\Phi_0$-equivariance just
means that $\tilde\alpha(\phi(x))$ additionally equals $1$ instead of
$-1$. Observe that applying instead of $\tilde\alpha$ a general
character $\chi\in\Lambda$ to equation \eqref{eq:1A} we get
\[
  \tilde\alpha(\phi(x))^{\<\chi,\Valpha^\vee\>}=1.
\]
Thus, if the root $\alpha$ satisfies $\<\Lambda,\Valpha^\vee\>=\ZZ$
then the condition \eqref{eq:atpx} is in fact superfluous.  This
phenomenon will be explored more carefully in section
\ref{sec:globalroot}.

Observe that all four conditions are local in $\cP$. Hence the
following makes sense:

\begin{definition}

  For any open subset $U\subseteq\cP$ let $\fL_{\cP,\Lambda}^\Phi(U)$
  be the set of of all smooth, closed, and $\Phi$-equivariant maps
  $\phi:U\to A$.

\end{definition}

Clearly, $\fL_{\cP,\Lambda}^\Phi$ is a sheaf of abelian groups on
$\cP$. Let $x\in\cP$. Then it is well known that the isotropy group
$W_x$ is the Weyl group of the (finite) root system
\[
  \Phi_x:=\{\alpha\in\Phi\mid\alpha(x)=0\}.
\]
Moreover, $W_y$ and $\Phi_y$ are contained in $W_x$ and $\Phi_x$,
respectively, for $y$ in a suitable open neighborhood $U$ of $x$. This
implies for the restriction to $U$:
\[
  \res_U\fL^\Phi_{\cP,\Lambda}=\res_U\fL^{\Phi_x}_{\cP,\Lambda}.
\]
This observation leads to the following generalization.

\begin{definition}

  A \emph{local system of roots $\Phi(*)$ on $\cP$} is a family
  $\big(\Phi(x)\big)_{x\in\cP}$ of root systems on $\fa$ such that for
  each $x\in \cP$:

  \begin{enumerate}

  \item\label{it:locroot1} $\Phi(y)=\Phi(x)_y$ for all $y$ in a
    sufficiently small neighborhood of $x$ in $\cP$.

  \item\label{it:locroot2} Every root $\alpha\in\Phi(x)$ is either
    non-negative or non-positive on $\cP$.

  \end{enumerate}

  An \emph{integral local system of roots on $\cP$} is a pair
  $(\Phi(*),\Lambda)$ such that $(\Phi(x),\Lambda)$ is a an integral
  root system for every $x\in\cP$.

\end{definition}

Observe that setting $y=x$ in condition \ref{it:locroot1} yields that
$\Phi(x)$ is centered at $x$, i.e., that $\alpha(x)=0$ for all
$\alpha\in\Phi(x)$. This implies in turn that all local root systems
$\Phi(x)$ are finite. The second condition \ref{it:locroot2} implies
that the set $\Phi^+(x)$ of roots which are non-negative on $\cP$
forms a set of positive roots for $\Phi(x)$. This allows to define a
set of simple roots
\[
  S(x)\subseteq \Phi^+(x)\subseteq\Phi(x).
\]

An main (and often the only) example of an integral local system of
roots on $\cP$ is $\Phi(x):=\Phi_x$ where $(\Phi,\Lambda)$ is an
integral root system on $\fa$ such that $\cP$ is entirely contained in
an alcove of $\Phi$. Such local systems are called \emph{trivial}. For
typical examples of trivial local root systems see figures
\eqref{eq:fig1}, \eqref{eq:fig2}, and \eqref{eq:fig3} where the gray
area is $\cP$ and the dashed lines denote the reflection hyperplanes.

Now it is easy to extend the definition of $\fL^\Phi_{\cP,\Lambda}$.

\begin{definition}

  Let $(\Phi(*),\Lambda)$ be an integral local system of roots on
  $\cP$. Then $\fL^{\Phi(*)}_{\cP,\Lambda}$ is the intersection of all
  $\fL^{\Phi(x)}_{\cP,\Lambda}$, $x\in\cP$, inside the sheaf of
  $A$-valued maps on $\cP$. Concretely, given $U\subseteq\cP$ open,
  then $\fL^{\Phi(*)}_{\cP,\Lambda}(U)$ is the set of all smooth,
  closed maps $\phi:U\to A$ which are $\Phi(x)$-equivariant for all
  $x\in U$.

\end{definition}

Observe that the coherence property \ref{it:locroot1} implies that
\[
  \res_U\fL^{\Phi(*)}_{\cP,\Lambda}=\res_U\fL^{\Phi(x)}_{\cP,\Lambda}.
\]
where $U$ is a suitable open neighborhood of $x\in\cP$.

Now let $(M,m)$ be a \mf\ $K\tau$-manifold. Let $m_+:M\to \cP_M$ be
its invariant moment map. For any $U\subseteq \cP_M$ open let
$M(U):=m_+^{-1}(U)$. This is again a \mf\ $K\tau$-manifold. Let
$\cAut_M$ be the sheaf of groups on $\cP_M$ defined by
$\cAut_M(U)=\Aut(M(U))$.

\begin{theorem}\label{thm:localroots}

  Let $K$ be a simply connected compact Lie group and $M$ be a convex
  \mf\ $K\tau$-manifold with momentum image $\cP=\cP_M$ and character
  group $\Lambda=\Lambda_M$. Then there is an integral local system of
  roots $(\Phi(*),\Lambda)$ on $\cP$ such that
  \[
    \cAut_M\cong\fL^{\Phi(*)}_{\cP,\Lambda}.
  \]

\end{theorem}

\begin{proof}

  Fix $a\in\cP$, let $u=\exp(x)$ be its image in $K$, and let $L$ be
  the centralizer of $u$ in $K$ (with respect to the twisted
  action). Then \cref{thm:cross} shows in particular that there exists
  an open neighborhood $U$ of $x\in\cP$ such that
  \[
    \cAut_{M(U)} = a+\cAut_{M_0}.
  \]
  Here, the left hand side denotes the sheaf of automorphisms of
  $M(U)=m_+^{-1}(U)$ as a \qH\ $K\tau$-manifold. The right hand side
  is the sheaf of automorphisms of $M_0=\log M_U$ as Hamiltonian
  $L$-manifold. The latter is a sheaf on $U-a$, so the ``$a+$''
  indicates translation back to $U$. Now the sheaf $\cAut M_0$ has
  been determined in \cite{KnopAuto}*{Thm. 9.2} with the result that
  there is a unique finite root system $\Phi(x)$, centered at $x$, and
  an isomorphism
  \[
    \fL^{\Phi(x)}_{U,\Lambda}\cong\cAut_{M_0}.
  \]
  Thereby, the root systems $\Phi(x)$ indeed form a local systems of
  roots by \cite{KnopAuto}*{eqn.~(9.4)}.

  This already shows that $\cAut_M$ and $\fL^{\Phi(*)}_{\cP,\Lambda}$
  are locally isomorphic. It remains to show that the isomorphisms
  \[
    \psi_U:\res_U\fL^{\Phi(*)}_{\cP,\Lambda}=\fL^{\Phi(x)}_{U,\Lambda}
    \overset\sim\to\res_U\cAut_M
  \]
  are compatible. For this we use that $\cP^0$, the interior of $\cP$,
  is dense in $\cP$. Thus, also $U^0=\cP^0\cap U$ is dense in $U$.
  Therefore it suffices to show that $\res_{U^0}\psi_U$ has a
  description which is independent of the choice of $x$ and $U$.

  To this end, observe that all roots $\alpha\in\Phi(x)$ are strictly
  positive on $\cP^0$ and therefore also on $U^0$. This implies that
  $\res_{U^0}\fL^{\Phi(x)}_{U,\Lambda}=\fL^\leer_{U^0,\Lambda}$ is
  just the sheaf of smooth closed maps $\phi:U^0\to A$. As mentioned
  above, all such $\phi$ are of the form $\phi=\exp\nabla f$ where $f$
  is a smooth function on $U^0$. Since $(m_0)_+$ and $m_+$ are smooth
  over $\cP^0$, the functions $F_0:=f\circ (m_0)_+$ and $F=f\circ m_+$
  are smooth on $M_{U^0}$ and $M(U^0)$, respectively. Moreover, it
  follows from \cite{KnopAuto}*{Thm.~9.1\emph{i)}} that $\phi$ is the
  flow of the Hamiltonian vector field $H_{F_0}$ at time $t=1$. The
  compatibility of the Hamiltonian vector fields $H_{F_0}$ and $H_F$
  (\cref{qHamFlow}) shows that $\psi_U(\phi)$ is the flow of $H_F$ at
  $t=1$. This shows that any two isomorphisms $\psi_U$ and $\psi_{U'}$
  coincide on $U\cap U'\cap\cP^0$ and therefore, by density, also on
  $U\cap U'$.
\end{proof}

\section{The global Weyl group}

In this and the following section we state and prove a criterion for
when an integral local system of roots $(\Phi(*),\Lambda)$ on a set
$\cP$ is trivial. For this consider first the system
$(W(x))_{x\in\cP}$ of Weyl groups of $\Phi(x)$. It forms a \emph{local
  system of reflection groups} in the sense that for all $x\in\cP$:
\begin{enumerate}

\item\label{it:localRef1} $W(y)=W(x)_y$ for all $y$ in a sufficiently
  small neighborhood of $x$ in $\cP$.

\item\label{it:localRef2} Let $s\in W(x)$ be a reflection with fixed
  point set $H_s$. Then $\cP$ lies entirely in one of the two closed
  halfspaces of $\fa$ which are defined by $H_s$.

\end{enumerate}

Observe that again the case $y=x$ of condition \ref{it:localRef1}
implies that $x$ is a fixed point of $W(x)$. In particular, all local
groups $W(x)$ are finite. In the following let $W\gl\subseteq M(\fa)$
be the subgroup generated by the union of all $W(x)$, $x\in\cP$. Our
aim is to prove that $W(x)=(W\gl)_x$ for all $x\in\cP$.

To this end, we first state a mostly classical criterion for when a
given set of reflections is the set of simple reflections for an
Euclidean reflection group.

\begin{lemma}\label{L1}

  Let $\alpha_1,\ldots,\alpha_n$ be non-constant affine linear
  functions on $\fa$ with:

  \begin{enumerate}

  \item\label{L1i1} For any $i\ne j$, the angle between $\Valpha_i$
    and $\Valpha_j$ equals $\pi-\frac\pi\ell$ with
    $\ell\in\ZZ_{\ge2}\cup\{\infty\}$.

  \item\label{L1i2} There is a point $x\in\fa$ with $\alpha_i(x)>0$
    for all $i=1,\ldots,n$.

  \end{enumerate}

\noindent
Let $W\subseteq M(\fa)$ be the group generated by the reflections
$s_{\alpha_1},\ldots,s_{\alpha_n}$. Then $W$ is an Euclidean
reflection group,
\[
  \cA:=\{x\in\fa\mid \alpha_1(x)\ge0,\ldots,\alpha_n(x)\ge0\}
\]
is an alcove for $W$, and the reflections
$s_{\alpha_1},\ldots,s_{\alpha_n}$ are precisely the simple reflection
with respect to $\cA$.

\end{lemma}

\begin{proof}

  Condition \ref{L1i2} implies that $\cA$ is a convex polyhedron with
  non-empty interior. Let, after renumbering,
  $\alpha_1,\ldots,\alpha_m$ be the non-redundant functions defining
  $\cA$, i.e., whose intersection $\{\alpha_i=0\}\cap\cA$ is of
  codimension $1$ in $\cA$. Then a classical theorem (see e.g.
  \cite{Vin}*{Thm.~1} for a much more general statement) asserts that,
  under condition \ref{L1i1}, $s_{\alpha_1},\ldots,s_{\alpha_m}$ are
  the simple reflections for an Euclidean reflection group $W$ and
  that $\cA$ is a fundamental domain. So, it remains to show that
  $m=n$. Suppose not. Then $\alpha_{m+1}$ would be redundant. This
  implies that there are real numbers $c_1,\ldots,c_m\ge0$ such that
  $\alpha_{m+1}=\sum_{i=1}^mc_i\alpha_i$. From \ref{L1i1} we get that
  \[
    \<\Valpha_i,\Valpha_{m+1}\>=
    \norm\Valpha_i\norm\norm\Valpha_{m+1}\norm\cdot\cos(\pi-\frac\pi\ell)\le0
  \]
  for $i=1,\ldots,m$ and therefore the contradiction
  $\<\Valpha_{m+1},\Valpha_{m+1}\>\le0$.
\end{proof}

Here is our criterion:

\begin{proposition}\label{L2}

  Let $\cP\subseteq\fa$ be a convex subset with non-empty interior and
  let $W(*)$ be a local system of reflection groups on $\cP$. Let
  $W\gl\subseteq M(\fa)$ be the group generated by the union of all
  $W(x)$, $x\in\cP$. Assume moreover that every $W\gl$-orbit meets
  $\cP$ in at most one point. Then $W\gl$ is an Euclidean reflection
  group with $W(x)=(W\gl)_x$ for all $x\in\cP$. Moreover, there is a
  unique alcove $\cA$ of $W\gl$ with $\cP\subseteq\cA$ and $\cP$ has a
  non-empty intersection with every wall of $\cA$.

\end{proposition}

\begin{proof}

  First we claim that
  \[\label{eq:e1}
    W(x)_y=W(y)_x\quad\text{for all $x,y\in\cP$}.
  \]
  Indeed, let $l=[x,y]\subseteq\fa$ be the line segment joining $x$
  and $y$. Then $l\subseteq\cP$ since $\cP$ is convex. For any
  $z\in l$ let
  \[
    W(z)_l:=\{w\in W(z)\mid wu=u\text{ for all $u\in l$}\}
  \]
  Then
  \[
    W(u)_l=(W(z)_u)_l=W(z)_l
  \]
  for all $u\in l$ which are sufficiently close to $z$. This means
  that the map $z\mapsto W(z)_l$ is locally constant, hence constant,
  on $l$. Thus
  \[
    W(x)_y=W(x)_l=W(y)_l=W(y)_x
  \]
  finishing the proof of the claim.

  Let $s=s_\alpha\in W\gl$ be a reflection with fixed point set
  $H:=\{\alpha=0\}$. We claim that $H$ does not meet $\cP^0$, the open
  interior of $\cP$. Otherwise, there would be points $x,y\in\cP^0$
  with $\alpha(x)>0$ and $\alpha(y)<0$. The line segment joining $x$
  and $y$ lies entirely in $\cP^0$ and meets $H$ in exactly one point
  $z$. Moreover there is $\epsilon>0$ such that both points
  $z_\pm:=z\pm\epsilon\Valpha$ are in $\cP^0$. But then $z_+$ and
  $z_-=s(z_+)$ would be two different points of $\cP$ lying in the
  same $W\gl$-orbit contradicting our assumption.

  The claim implies that $\cP^0$, being connected, lies entirely in
  one of the open halfspaces determined by $H$. Hence $\cP$, its
  closure, lies entirely in one of the two closed halfspaces
  determined by $H$.

  This reasoning applies, in particular, to all reflections contained
  in $W(x)$. Thus, $\cP$ is contained in a unique Weyl chamber
  $C(x)\subseteq\fa$ for $W(x)$. This chamber determines in turn a set
  $S(x)\subset W(x)$ of simple reflections. It is well-known that for
  any $y\in C(x)$ the set $S(x)_y:=\{s\in S(x)\mid sy=y\}$ is a set of
  simple reflections for $W(x)_y$. Therefore equation \eqref{eq:e1}
  implies that
  \[
    S(x)_y=S(y)_x\quad\text{for all $x,y\in\cP$}.
  \]
  Now let $S$ be the union of all $S(x)$, $x\in\cP$. Then
  \[\label{eq:Sx}
    S(x)=\{s\in S\mid sx=x\}
  \]
  for all $x\in\cP$. Indeed, let $s\in S$ with $sx=x$. Then
  $s\in S(y)$ for some $y\in\cP$. Thus,
  $s\in S(y)_x=S(x)_y\subseteq S(x)$.

  For each $s\in S$ choose affine linear functions $\alpha_s$ with
  $s=s_{\alpha_s}$ and such that $\alpha_s\ge0$ on $\cP$. We are going
  to show that $\{\alpha_s\mid s\in S\}$ satisfies the assumptions of
  Lemma \ref{L1}.

  Let $s_1\ne s_2\in S$. Put $\alpha_i:=\alpha_{s_i}$ and
  $H_i:=\{\alpha_i=0\}$. Assume first that $H_1$ and $H_2$ are
  parallel. Then $\Valpha_1=c\Valpha_2$ with $c\ne0$ and we have to
  show that $c<0$. The functions $\alpha_i$ vanish, by construction,
  at some points $x_i\in\cP$. Put $t:=x_1-x_2\in\Vfa$. Then
  $\<\Valpha_1,t\>=-\alpha_1(x_2)<0$ and
  $\<\Valpha_2,t\>=\alpha_2(x_1)>0$ which shows $c<0$.

  Now assume that $H_1$ and $H_2$ are not parallel. Then
  $E:=H_1\cap H_2$ is a subspace of codimension two. Let
  $W'\subseteq W$ be the dihedral group generated by $s_1$ and $s_2$
  and let $\theta$ be the angle between $\Valpha_1$ and
  $\Valpha_2$. Then $W'$ contains the rotation $r$ around $E$ with
  angle $2\theta$. If $r$ had infinite order then the union of all
  $\<r\>$-translates of, say, $H_1$ would be dense in $\fa$. But that
  contradicts the assumption that every $W\gl$-orbit meets $\cP$ at
  most once. Therefore $W'$ is a finite reflection group.

  Now we claim that $\{s_1,s_2\}$ forms a set of simple reflections
  for $W'$. If $E\cap\cP\ne\leer$ this is clear since then
  $s_1,s_2\in S(x)$ for all $x\in E\cap\cP$ (by
  eqn.~\eqref{eq:Sx}). So assume $E\cap\cP=\leer$. Let $C'$ be the
  unique Weyl chamber of $W'$ which contains $\cP$ and let
  $s_i'\in W'$, $i=1,2$, be the corresponding simple
  reflections. Choose functions $\alpha_i'$ with $s_i'=s_{\alpha_i'}$
  such that $\alpha_i'\ge0$ on $\cP$. Observe that
  \[
    E=\{\alpha_1=\alpha_2=0\}=\{\alpha_1'=\alpha_2'=0\}=\fa^{W'}.
  \]
  Now fix $i\in\{1,2\}$. Then $\alpha_i=c_1\alpha_1'+c_2\alpha_2'$ for
  some real numbers $c_1,c_2\ge0$. Suppose $c_1,c_2>0$, i.e., $s_i$ is
  not simple. By construction $s_i\in W(x)$ for some $x\in\cP$. Then
  \[
    0=\alpha_i(x)=c_1\alpha_1'(x)+c_2\alpha_2'(x)
  \]
  implies $\alpha_1'(x)=\alpha_2'(x)=0$ and therefore $x\in\cP\cap E$
  which is excluded.

  The fact that $s_1$ and $s_2$ are simple reflections of $W'$ implies
  that the angle of $\Valpha_1$ and $\Valpha_2$ is of the form
  $\pi-\frac\pi\ell$ with $\ell\in\ZZ_{\ge2}\cup\{\infty\}$.  Since
  condition \ref{L1i2} is obvious from $\cP^0\subseteq\cA$ we can
  apply \cref{L1} and infer that $W\gl$ is an Euclidean reflection
  group with alcove $\cA\supseteq\cP$ and that $S$ is a set of simple
  reflections of $W$. Finally, \eqref{eq:Sx} implies
  \[
    W_x=\langle s\in S\mid sx=x\rangle=\langle S(x)\rangle=W(x).
  \]
  for all $x\in\cP$. The last assertion holds by construction.
\end{proof}

Applying this to multiplicity free manifolds we get:

\begin{corollary}\label{cor:globalWeyl}

  Let $\Phi(*)$ be the local system of roots of a convex \mf\
  $K\tau$-manifold (as in \cref{thm:localroots}). Let $W(*)$ be the
  corresponding local system of reflection groups. Then there is an
  Euclidean reflection group $W_M\subseteq M(\fa_M)$ such that
  $W(x)=(W_M)_x$ for all $x\in\cP_M$.

\end{corollary}

\begin{proof}

  This follows from \cref{L2} as soon as we have shown that each
  $W\gl$-orbit meets $\cP_M$ in at most one point. To this end recall
  that the Weyl group $W_\tau$ of the root system $\Phi_\tau$ (see
  \cref{thm:conjugacyclasses} acts on the Cartan subspace $\fa$ as an
  Euclidean reflection group. Let $N$ and $C$ be the normalizer and
  centralizer, respectively, of $\fa_M\subseteq\fa$. Then
  $\tilde W:=N/C$ acts properly on $\fa_M$ and each $\tilde W$-meets
  $\fa_M\cap\cP$ at most once. Now the claim follows from the fact
  that each local Weyl group $W(x)$, hence $W\gl$, is a subgroup of
  $\tilde W$ (see, e.g., \cite{KnopAuto}*{Thm.~3.2}).
\end{proof}

\section{The global root system}\label{sec:globalroot}

In this section, we complete the proof that the local system of roots
attached to a multiplicity free manifold is trivial. We already know
that there is a global Weyl group and a weight lattice. So for the
root system there are only finitely many possibilities which we are
first going to investigate.

More abstractly, let $W\subseteq M(\fa)$ be an Euclidean reflection
group, $\Lambda\subseteq\Vfa$ a weight lattice for $W$ (see
\cref{def:weight}), and $\cA\subseteq\fa$ an alcove. Let
$S\subseteq W$ be the set of simple reflections with respect to $\cA$.

For $s\in S$ let
\[
  \Lambda^{\pm s}:=\{\chi\in\Lambda\mid s(\chi)=\pm\chi\}.
\]
Since $\Lambda^{-s}\cong\ZZ$ there is a unique affine function
$\alpha_s\prim$ on $\fa$ with $\fa^{\<s\>}=\{\alpha_s\prim=0\}$, which
is non-negative on $\cA$ and such that $\Valpha_s\prim$ is a generator
of $\Lambda^{-s}$. If $\alpha_s$ is a root for $s$ in some root system
then necessarily $\alpha_s=\alpha_s\prim$ or
$\alpha_s=2\alpha_s\prim$. In the second case
$\Valpha_s^\vee=\half(\Valpha_s\prim)^\vee$ which therefore can only
occur if $\<\Lambda,(\Valpha_s\prim)^\vee\>=2\ZZ$. There is another
way to put this: $\<\Valpha_s\prim,\Valpha_s^\vee\>=1$ implies that
$\Lambda^{-s}$ is an $\<s\>$-equivariant direct summand of $\Lambda$.

\begin{definition}

  \emph{a)} A simple reflection $s\in S$ is called {\it ambiguous} if
  $ \<\Lambda,(\Valpha_s\prim)^\vee\>=2\ZZ$ or, equivalently, if
  $\Lambda=\Lambda^{+s}\oplus\Lambda^{-s}$. The set of ambiguous
  simple reflections is denoted by $\Samb\subseteq S$.

  \emph{b)} Let $(\Phi,\Lambda)$ be an integral root system with Weyl
  group $W_\Phi=W$. Then
  \[
    \Samb(\Phi):=\{s\in S\mid\alpha_s=2\alpha\prim_s\}\subseteq \Samb.
  \]

\end{definition}

\begin{remarks}

  \emph{a)} The notion of ambiguity depends on the lattice $W$. Let,
  e.g., $W$ be a Weyl group of type $\sB_n$ and let $s$ be the
  reflection corresponding to the short simple root. Then $s$ is
  ambiguous with respect to the roots lattice but not with respect to
  the weight lattice.

  \emph{b)} Some simple roots are never ambiguous. Assume for example
  that there is a simple root $s'\in S$ with $(ss')^3=1$. This means
  that $s$ and $s'$ are joined by a simple edge in the Coxeter diagram
  of $W$. Then
  \[
    \<\Valpha_{s'}\prim,(\Valpha_s\prim)^\vee\>=-1
  \]
  implies that $s$ is not ambiguous.

\end{remarks}

Now we classify all root systems $\Phi$ with a given weight lattice
$\Lambda$ and Weyl group $W$.

\begin{lemma}\label{L3}

  Let $W\subseteq M(\fa)$ be an Euclidean reflection group, let
  $\cA\subseteq\fa$ be an alcove of $W$, let $S\subseteq W$ be the
  corresponding set of simple reflections, and let
  $\Lambda\subseteq\Vfa$ be a weight lattice. Then:

  \begin{enumerate}

  \item\label{L3i1} No two distinct elements of $\Samb$ are conjugate
    within $W$.

  \item\label{L3i2} The map $\Phi\mapsto \Samb(\Phi)\subseteq\Samb$ is
    a bijection between root systems $\Phi$ with Weyl group $W$ and
    weight lattice $\Lambda$, and subsets of $\Samb$.

  \end{enumerate}

\end{lemma}

\begin{proof}

  \ref{L3i1} It is well-known (see e.g. \cite{Bou} IV, \S1, Prop. 3)
  that two simple reflections $s,s'$ in a Coxeter group are conjugate
  if and only if there is a string of simple reflections
  $s=s_1,s_2,\ldots,s_n=s'$ such that the order of $s_is_{i+1}$ is odd
  for all $i=1,\ldots,n-1$. For Weyl groups this happens only if the
  order is $3$. But then, by the Remark~\emph{b)} above, neither $s$
  nor $s'$ is ambiguous.

  \ref{L3i2} We construct the inverse mapping. For $I\subseteq \Samb$
  let
  \[
    \alpha_s^I:=\begin{cases} 2\alpha\prim_s&\text{if }s\in
      I\\\alpha\prim_s&\text{if }s\in S\setminus I
    \end{cases}
  \]
  and $S_I:=\{\alpha_s^I\mid s\in S\}$. Then part \ref{L3i1} shows
  that $\Phi_I=WS_I$ is a {\it reduced} root system with simple roots
  $S_I$ and $\Samb(\Phi_I)=I$.
\end{proof}


Next, we determine the set $\Samb$ of ambiguous roots.

\begin{proposition}\label{P2}

  Let $W$, $\cA$, $S$, and $\Lambda$ be as in \cref{L3}.

  \begin{enumerate}

  \item\label{P2i1} Reduction to irreducible reflection groups: Let
    $S_0\subseteq S$ be a connected component of the Coxeter diagram
    of $S$ containing an ambiguous root. Let
    $\Vfa_0:=\<\Valpha_s\prim\in\Vfa\mid s\in S_0\>_\RR$ and
    $\Lambda_0:=\Lambda\cap\Vfa_0$. Then $\Lambda_0$ is a
    $W$-equivariant direct summand of $\Lambda$.

  \item\label{P2i2} Assume $(W,S)$ to be irreducible with
    $\Samb\ne\leer$. Then $\Lambda=\<\Valpha\prim_s\mid s\in
    S\>_\ZZ$. In particular, $(W,\Lambda)$ is determined by the root
    system $\Phi_\leer$ (see proof of \cref{L3} for the notation).

  \item\label{P2i3} The irreducible root systems which are of the form
    $\Phi_\leer$ are listed in the following table. The set $\Samb$ is
    marked by asterisks.

  \end{enumerate}

  \def\scalar{0.9}

  \setlength{\unitlength}{1mm}
  \begin{equation*}
    \begin{array}{ll|l}
      \Phi_\leer&&\text{Diagram}\\
      \hline
      \sA_1&&
              \begin{tikzpicture}[scale=1]
                \Ddot(1,0;*)
              \end{tikzpicture}
      \\
      \sB_n&(n\ge2)&
                     \begin{tikzpicture}[scale=1]
                       \Ddot(1,0;) \Ddot(2,0;) \Ddot(4,0;) \Ddot(5,0;)
                       \Ddot(6,0;*) \Dline(1,0) \Ddots(2,0)
                       \Dline(4,0) \Drarrow(5,0)
                     \end{tikzpicture}
      \\
      \sA_1^{(1)}&&
                    \begin{tikzpicture}[scale=1]
                      \Ddot(1,0;*) \Ddot(2,0;*) \Dlrarrow(1,0)
                    \end{tikzpicture}
      \\
      \sB_2^{(1)}&&
                    \begin{tikzpicture}[scale=1]
                      \Ddot(1,0;) \Ddot(2,0;*) \Ddot(3,0;)
                      \Drarrow(1,0) \Dlarrow(3,0)
                    \end{tikzpicture}
      \\
      \sB_n^{(1)}&(n\ge3)&
                           \vcenter{\hbox{%
                           \begin{tikzpicture}[scale=1]
                             \Ddot(1,0.5;) \Ddot(1,-0.5;) \Ddot(2,0;)
                             \Ddot(3,0;) \Ddot(5,0;) \Ddot(6,0;)
                             \Ddot(7,0;*)
                             \draw[thick](2,0)--+(-1,0.5);
                             \draw[thick](2,0)--+(-1,-0.5);
                             \Dline(2,0) \Ddots(3,0) \Dline(5,0)
                             \Drarrow(6,0)
                           \end{tikzpicture}
                           }}
      \\
      \vrule width0pt height25pt\sD_{n+1}^{(2)}&(n\ge2)&
                                                         \begin{tikzpicture}[scale=1]
                                                           \Ddot(1,0;*)
                                                           \Ddot(2,0;)
                                                           \Ddot(3,0;)
                                                           \Ddot(5,0;)
                                                           \Ddot(6,0;)
                                                           \Ddot(7,0;*)
                                                           \Dlarrow(2,0)
                                                           \Dline(2,0)
                                                           \Ddots(3,0)
                                                           \Dline(5,0)
                                                           \Drarrow(6,0)
                                                         \end{tikzpicture}
      \\
    \end{array}
  \end{equation*}

\end{proposition}

\begin{proof}

  We first prove \ref{P2i2}. Assume therefore that $(W,S)$ is
  irreducible and that $\Phi=\Phi_\leer$ for some choice of
  $\Lambda$. Let $s\in \Samb$. Then $\<\beta,\Valpha_s^\vee\>$ must be
  even for all $\beta\in S$. The classification of (affine) root
  systems (see e.g. \cite{Kac} or \cite{Mac2}) shows that the Dynkin
  diagram of $\Phi$ is one of the items in the table above or is one
  of the following
  \[\label{eq:A2n1}
    \begin{array}{ll|l}
      \sA_2^{(2)}&&
                    \begin{tikzpicture}[scale=1]
                      \draw(1,0)
                      node[above]{$\beta$}; \Ddot(1,0;) \Ddot(2,0;*)
                      \Drarrow(1,0) \draw[thick] (1,3pt)--+(0.7,0);
                      \draw[thick] (1,-3pt)--+(0.7,0);
                    \end{tikzpicture}\\
      \sA_{2n}^{(2)}&(n\ge2)&
                              \begin{tikzpicture}[scale=1]\
                                \draw(1,0) node[above]{$\beta$};
                                \Ddot(1,0;)
                                \Ddot(2,0;)
                                \Ddot(3,0;)
                                \Ddot(5,0;)
                                \Ddot(6,0;)
                                \Ddot(7,0;*)
                                \Drarrow(1,0)
                                \Dline(2,0)
                                \Ddots(3,0)
                                \Dline(5,0)
                                \Drarrow(6,0)
                              \end{tikzpicture}\\
    \end{array}
  \]
  Moreover, $s$ corresponds to one of vertices marked by an asterisk.
  Now from $\Valpha_s\in\Lambda$ and
  $\gamma:=\frac12\Valpha_s^\vee\in\Lambda^\vee$ we get
  \[\label{e6}
    \<W\Valpha_s\>\subseteq\Lambda\subseteq \<W\gamma\>^\vee.
  \]
  It is easy to verify case-by-case that $W\Valpha_s$ equals either
  $\{\epsilon_1,\ldots,\epsilon_n\}$ or
  $\{\pm\epsilon_1,\ldots,\pm\epsilon_n\}$ where
  $\epsilon_1,\ldots,\epsilon_n$ is an orthogonal basis of $\Vfa$ all
  whose elements have the same length. This and
  $\<\Valpha_s,\gamma\>=1$ immediately imply that the outer lattices
  of \eqref{e6} coincide. Hence, $\Lambda$ equals the root lattice of
  $\Phi$ proving \ref{P2i2}.
  
  From this also \ref{P2i1} follows: Let $\tilde\Lambda_0$ be the
  orthogonal projection of $\Lambda$ to $\Valpha_0$. Then the
  inclusions \ref{e6} hold for both $\Lambda_0$ and
  $\tilde\Lambda_0$. So $\Lambda_0=\tilde\Lambda_0$ which implies that
  $\Lambda_0$ is a $W$-direct summand of $\Lambda$.

  Finally, for \ref{P2i3}, observe that the root systems
  $\sA_{2n}^{(2)}\ (n\ge1)$ cannot be of the form $\Phi_\leer$ since
  $s_\beta\in \Samb(\Phi)$.
\end{proof}

Next, we prove a refinement of \cref{L2} to root systems.

\begin{proposition}\label{P1}

  Let $\cP\subseteq\fa$ be a convex subset with non-empty interior and
  let $(\Phi(*),\Lambda)$ be an integral local system of roots on
  $\cP$. Let $W$ be the group generated by all the local Weyl groups
  $W(x)$, $x\in\cP$, and assume that every $W$-orbit meets $\cP$ at
  most once. Then there is an integral root system $(\Phi,\Lambda)$
  with $\Phi(x)=\Phi_x$ for all $x\in\cP$.

\end{proposition}

\begin{proof}

  \cref{L2} implies that $W$ is an Euclidean reflection group with
  $W(x)=W_x$ for all $x\in\cP$. In particular, $(\Phi(x),\Lambda)$ is
  an integral root system with Weyl group $W_x$. Let $S\subseteq W$ be
  the set of simple reflections. Then $S_x=S\cap W_x$ is a set of
  simple roots for $\Phi(x)$. Moreover, $\Phi(x)$ is a root system
  with Weyl group $W_x$ such that $\Lambda$ is a weight lattice and
  \[
    \Samb(x):=\{s\in S_x\mid
    \Valpha_s\prim\not\in\Phi(x)\}\subseteq\Samb
  \]
  Now the same argument as for \eqref{eq:e1} also shows
  \[\label{eq:phiphi}
    \Phi(x)_y=\Phi(y)_x\text{ for all $x,y\in\cP$}.
  \]
  This implies that whenever $s\in S_x\cap S_y$ then $s\in\Samb(x)$ if
  and only if $\Samb(y)$. Thus, the union
  $S_0=\bigcup_{x\in\cP}\Samb(x)$ has the property that
  $S_0\cap S_x=\Samb(x)$ for all $x$. Therefore the root system $\Phi$
  attached to $S_0\subseteq\Samb$ satisfies $\Phi_x=\Phi(x)$, as
  required.
\end{proof}

Now we can improve on \cref{thm:localroots}:

\begin{theorem}\label{thm:globalroots}

  Let $K$ be a simply connected compact Lie group with twist $\tau$
  and $M$ be a convex \mf\ $K\tau$-manifold.  Then there is a unique
  affine root system $\Phi_M$ on $\fa_M$ such that
  \begin{itemize}

  \item $\Lambda_M$ is a weight lattice for $\Phi_M$.

  \item $\cP_M$ is contained in a (unique) alcove $\cA$ of $\Phi_M$.

  \item $\cP_M$ intersects every wall of $\cA$.

  \item The sheaf of automorphisms $\cAut_M$ is canonically isomorphic
    to $\fL^{\Phi_M}_{\cP_M,\Lambda_M}$.

  \end{itemize}

\end{theorem}

\begin{proof}

  Apply \cref{P1} to \cref{thm:globalroots}. The condition on $W$ has
  been verified in the proof of \cref{cor:globalWeyl}.
\end{proof}

\section{Cohomology computations}\label{sec:cohomology}

In this section, we provide the last step of the proof of our main
classification \cref{thm:main}.

\begin{theorem}\label{T1}

  Let $(\Phi,\Lambda)$ be an integral root system on the Euclidean
  affine space $\fa$, let $\cA$ be a fixed alcove of $\Phi$, and let
  $\cP\subseteq\cA$ be a locally closed convex subset with non-empty
  interior. Then $H^i(\cP,\fL^\Phi_{\cP,\Lambda})=0$ for all $i\ge1$.

\end{theorem}

The proof will occupy the rest of this section. We start with a
reduction step:

\begin{lemma}\label{lem:commensurable}

  Let $\Lambda_1,\Lambda_2\subseteq\Vfa$ be two commensurable weight
  lattices for $\Phi$. Then
  $H^i(\cP,\fL^\Phi_{\cP,\Lambda_1})=H^i(\cP,\fL^\Phi_{\cP,\Lambda_1})$
  for all $i\ge1$

\end{lemma}

\begin{proof}

  By replacing $\Lambda_1$ with the intersection
  $\Lambda_1\cap\Lambda_2$ we may assume
  $\Lambda_1\subseteq\Lambda_2$. Then $A_1:=\Vfa/\Lambda_1^\vee$ is a
  quotient of $A_2:=\Vfa/\Lambda_2^\vee$ with kernel
  \[\label{e3}
    E:=\Lambda_1^\vee/\Lambda_2^\vee\subseteq A_2^\Phi.
  \]
  Let $U\subseteq\cP$ be convex and open. Then any map
  $\phi_1:U\to A_1$ can be lifted to a map $\phi_2:\cP\to
  A_2$. Moreover, $\phi_2$ is smooth, closed, and $\Phi$-invariant if
  and only $\phi_1$ is. Thus, we get a short exact sequence of sheaves
  \[\label{e4}
    0\Pfeil
    E_\cP\to\fL^\Phi_{\cP,\Lambda_2}\Pfeil\fL^\Phi_{\cP,\Lambda_1}\Pfeil0
  \]
  where $E_\cP$ denotes the constant sheaf on $\cP$ with fiber
  $E$. Since $\cP$ is convex, we have $H^i(\cP,E_\cP)=0$ for
  $i\ge1$. From this the assertion follows.
\end{proof}

A weight lattice will be called \emph{of adjoint type} if
\[
  \Lambda=\ZZ\VPhi\oplus\Lambda^W\subseteq\RR\VPhi\oplus\Vfa^W=\Vfa.
\]
Since every weight lattice $\Lambda$ is commensurable to
$\ZZ\VPhi\oplus\Lambda^W$ the Lemma allows us to assume that $\Lambda$
is of adjoint type.

Now recall that the sections of $\fL^\Phi_{\cP,\Lambda}$ are the
smooth, closed, $\Phi$-equivariant maps $\phi:U\to A$ where $A$ is the
torus $\Vfa/\Lambda^\vee$ and $U\subseteq\cP$ is open. To construct
maps of this type, consider a smooth function $f$ defined on $U$. As
explained in section~\ref{sec:auto}, the map
\[
  \epsilon(f):U\to A:x\mapsto \exp(2\pi\nabla f(x))
\]
is smooth and closed. It is $\Phi$-equivariant whenever $f$ is
\emph{$W$-invariant} in the sense that for each $x\in\cP$ the Taylor
series $\hat f$ of $f$ in $x$ in $W_x$-invariant. Let $\cC_\cP^W$ be
the sheaf of $W$-invariant smooth functions on $\cP$. This way, we get
a homomorphism of sheaves
\[\label{eq:epsilon}
  \epsilon:\cC_\cP^W\to\fL_{\cP,\Lambda}^\Phi.
\]

Our first goal is to study the cokernel of this map. To this end,
consider the subgroup
\[
  A^\Phi:=\{u\in A\mid \tilde\alpha(u)=1\text{ for all
    $\alpha\in\Phi$}\}.
\]
Its elements are called the \emph{$\Phi$-fixed points} of $A$.  By
\eqref{e2}, they form a subgroup of $A^W$, the group of $W$-fixed
points.  Of particular interest will be the component group
$\pi_0(A^{\Phi_x})$.

One can localize these constructions. For any $x\in\cP$ consider the
groups $A^{\Phi_x}$ and $\pi_0(A^{\Phi_x})$. If $y$ is close to $x$
then $\Phi_y\subseteq\Phi_x$ and therefore
\[\label{eq:inclusion}
  A^{\Phi_x}\subseteq A^{\Phi_y}.
\]
This shows that there is a constructible sheaf $\fC_\cP$ such that
$\pi_0(A^{\Phi_x})$ is the stalk at $x$ and the restriction maps
$\pi_0(A^{\Phi_x})\to\pi_0(A^{\Phi_y})$ are induced by
\eqref{eq:inclusion}. It significance is given by

\begin{lemma}

  There is an exact sequence
  \[
    \cC_\cP^W\overset\epsilon\to\fL_{\cP,\Lambda}^\Phi\overset\eta\to\fC_\cP\to0.
  \]

\end{lemma}

\begin{proof}

  Let $U\subseteq\cP$ be open and fix $x\in U$. If
  $\phi\in\fL^\Phi_{\cP,A}(U)$ then $\phi(x)\in A^{\Phi_x}$, by
  $\Phi$-equivariance. Thus we can define $\eta(\phi)(x)$ to be the
  image of $\phi(x)$ in $\pi_0(A^{\Phi_x})$.

  Now let $u\in A^{\Phi_x}$ be a representative of some element
  $\uq\in\pi_0(A^{\Phi_x})$. Then the constant map $\phi:x\mapsto u$
  is a section of $\fL^\Phi_{\cP,A}$ with $\eta(\phi)=\uq$. This shows
  that $\eta$ is surjective.

  On the other hand, for any section $f$ of $\cC_\cP^W$, the image
  $\epsilon(f)(x)$ lies in $\exp(\Vfa^{W_x})=(A^{\Phi_x})^0$. This
  shows $\|im|\epsilon\subseteq\ker\eta$.

  To show equality, let $\phi:U\to A$ be a section of
  $\fL^\Phi_{\cP,A}$ with $\eta(\phi)=0$, i.e.,
  $\phi(x)\in(A^{\Phi_x})^0$ for all $x$. Then there is a lift
  $\tilde\phi:U\to\Vfa$ of $\phi$ with
  $\tilde\phi(x)\in\Vfa^{W_x}$. The last condition implies that
  $\tilde\phi$ is $W_x$-equivariant. Since $\tilde\phi$ is smooth and
  closed there is a smooth function $f$ on $U$ is with
  $\nabla f=\tilde\phi$. Let $\hat f$ be the Taylor series of $f$ in
  $x$. Then the $W_x$-equivariance of $\tilde\phi$ implies
  $\nabla\,{}^w\hat f=\nabla\hat f$ for all $w\in W_x$. Hence
  $c_w:={}^w\hat f-\hat f$ is a constant and $w\mapsto c_w$ is
  homomorphism. We conclude $c_w=0$, i.e., $f$ is in fact
  $W_x$-invariant. Therefore, $\phi=\epsilon(f)$ is indeed in the
  image of $\epsilon$
\end{proof}

To investigate the cohomology of $\fC_\cP$ we need a more explicit
description. The character group of $A^\Phi$ is given by
\[
  \Xi(A^\Phi)=\Lambda/\ZZ\VPhi.
\]
In particular, $\pi_0(A^\Phi)=0$ if and only if the root lattice
$\ZZ\VPhi$ is a direct summand of $\Lambda$. More generally, we have
\[
  \Xi(\pi_0(A^\Phi))=\|Tors|(\Lambda/\ZZ\VPhi)
  =\frac{\Lambda\cap\RR\VPhi}{\ZZ\VPhi}.
\]
Dualizing, this is equivalent to
\[\label{eq:pi0ApHi}
  \pi_0(A^\Phi)=\frac{(\ZZ\VPhi)^\vee}{\Lambda^\vee+(\RR\VPhi)^\vee}
\]
where $(\ZZ\VPhi)^\vee$ is the coweight lattice and $(\RR\VPhi)^\vee$
is the orthogonal complement of $\RR\Phi$.

We compute $\fC_\cP$ in two stages, the first being the case of finite
root systems.

\begin{lemma}\label{lem:finiteC}

  Assume $\Phi$ is finite and $\Lambda=\ZZ\VPhi$. Then $\fC_\cP=0$.

\end{lemma}

\begin{proof}

  Let $S=\{\alpha_1,\ldots,\alpha_n\}$ be the set of simple roots of
  $\VPhi$. Since these form a basis of $\Vfa$ we get an isomorphism
  \[\label{eq:alpha*}
    \alpha_*:\Vfa\to\RR^n:x\mapsto(\<\alpha_1,x\>,\ldots,\<\alpha_n,x\>)
  \]
  For any subset $I\subseteq\{1,\ldots,n\}$ let $I'$ be its
  complement. Moreover, for $k\in\{\RR,\ZZ\}$ we put
  \[
    k^I:=\{(x_i)\in\RR^n\mid x_i\in k\text{ for }i\in I\text{ and
      $x_i=0$ otherwise}\}\cong k^{|I|}
  \]
  For any $x\in\cP$ let $I:=\{i\mid\alpha_i(x)=0\}$. Then $\alpha_*$
  maps $(\ZZ\VPhi_x)^\vee$, $(\RR\VPhi_x)^\perp$, and $\Lambda^\vee$
  to $\ZZ^I\oplus\RR^{I'}$, $\RR^{I'}$, and $\ZZ^n$, respectively. Now
  the claim follows from \eqref{eq:pi0ApHi}.
\end{proof}

Now suppose $\Phi$ is an infinite irreducible root system with simple
roots $S=\{\alpha_1,\ldots,\alpha_n\}$. The \emph{labels} of $S$ are
defined as the components of the unique primitive vector
$(a_1,\ldots,a_n)\in\ZZ_{>0}^n$ such that
\[
  a_1\Valpha_1+\ldots+a_n\Valpha_n=0.
\]
For $I\subsetneq\{1,\ldots,n\}$ we define $I'\ne\leer$ as its
complement and
\[
  d_I:=\operatorname{gcd}\{a_j\mid j\in I'\}.
\]

\begin{lemma}\label{lem:infiniteC}

  Assume $\Phi$ is an infinite irreducible root system and
  $\Lambda=\ZZ\VPhi$. For any $x\in\cA$ let
  $I:=\{i\mid\alpha_i(x)=0\}$. Then there is a canonical isomorphism
  \[
    \pi_0(A^{\Phi_x})\to\ZZ/d_I\ZZ.
  \]
  Moreover, this isomorphism is compatible with the restriction
  homomorphisms of $\fC_\cP$.

\end{lemma}

\begin{proof}

  We keep the notation of the proof of \cref{lem:finiteC}. Let
  $\delta:=(a_1,\ldots,a_n)\in\RR^n$ be the vector of labels. Then the
  map $\alpha_*$ of \eqref{eq:alpha*} identifies $\Vfa$ with the
  hyperplane $H$ of $\RR^n$ which is perpendicular to $\delta$. Thus
  \eqref{eq:pi0ApHi} becomes
  \[
    \pi_0(A^{\Phi_x})=\frac{(\ZZ^I\oplus\RR^{I'})\cap H}
    {(\RR^{I'}\cap H)+(\ZZ^n\cap H)}
  \]
  Now consider the homomorphism
  \[
    p_I:(\ZZ^I\oplus\RR^{I'})\cap
    H\to\ZZ/d_I\ZZ:(x_i)\mapsto\sum_{i\in I}a_ix_i+d_I\ZZ.
  \]
  Since $d_{I'}$ and $d_I$ are coprime there are $a',a\in\ZZ$ with
  $a'd_{I'}+ad_I=1$. Because $I'\ne\leer$, there is
  $(x_i)\in(\ZZ^I\oplus\RR^{I'})\cap H$ with
  $\sum_{i\in I}a_ix_i=a'd_{I'}$. Then $p_I(x_i)=1$, i.e., $p_I$ is
  onto.

  Next we claim that the kernel of $p_I$ is precisely
  $E:=(\RR^{I'}\cap H)+(\ZZ^n\cap H)$. Clearly
  $\RR^{I'}\cap H\subseteq\ker p_I$. Let $(x_i)\in\ZZ^n\cap H$. Then
  \[
    \sum_{i\in I}a_i x_i=-\sum_{j\in I'}a_jx_j\in d_I\ZZ
  \]
  shows that $E\subseteq\ker p_I$. To show the converse, let
  $(x_i)\in\ker p_I$. Then, by definition,
  $\sum_{i\in I}a_ix_i\in d_I\ZZ$. Hence there is $(y_i)\in\ZZ^{I'}$
  with
  \[
    \sum_{i\in I}a_i x_i=-\sum_{j\in I'}a_jy_j.
  \]
  Now define
  \[
    \overline x_i:=
    \begin{cases}
      x_i&\text{if }i\in I\\
      y_i&\text{if }i\in I'
    \end{cases}
  \]
  Then $(\overline x_i)\in\ZZ^n\cap H$ with
  $(x_i)-(\overline x_i)\in\RR^{I'}\cap H$, proving the claim. Thus
  $p_I$ induces an isomorphism between $\pi_0(A^{\Phi_x})$ and
  $\ZZ/d_I\ZZ$.

  For the final claim, we denote $I$ by $I_x$. Let $y\in\cP$ be close
  to $x$. Then $I_y\subseteq I_x$ and therefore
  $d_{I_y}|d_{I_x}$. Thus, we have to show that the diagram
  \[
    \cxymatrix{ (\ZZ^{I_x}\oplus\RR^{I_x'})\cap
      H\ar[r]^>>>{p_{I_x}}\ar@{^(->}[d]&
      \ZZ/d_{I_x}\ZZ\ar@{>>}[d]^{[1]\mapsto[1]}\\
      (\ZZ^{I_y}\oplus\RR^{I_y'})\cap
      H\ar[r]^>>>{p_{I_y}}&\ZZ/d_{I_y}\ZZ}
  \]
  commutes. But this follows from $d_{I_y}|\,a_i$ for all
  $i\in I_x\setminus I_y$.
\end{proof}

From this we deduce:

\begin{lemma}

  Assume $\Lambda$ is of adjoint type. Then $\fC_\cP$ has a finite
  filtration such that each factor is a constant sheaf supported on a
  face of $\cP$.

\end{lemma}

\begin{proof}

  Let $\fa=\fa_0\times\fa_1\times\ldots\times\fa_m$ and
  $\Phi=\Phi_1\cup\ldots\cup\Phi_m$ be the unique decomposition of
  $(\fa,\Phi)$ into a trivial part $\fa_0$ and irreducible parts
  $\fa_1,\ldots,\fa_m$. Then $\fC_P=\fC^{(1)}\oplus\ldots\fC^{(m)}$
  where $\fC^{(i)}$ is the pull-back of $\fC_{\cA_i}$ to $\cP$. Thus
  it suffices to show the assertion for $\fC:=\fC^{(i)}$ for any
  $i$. We may also assume that $\Phi_i$ is infinite. Let
  $\alpha_1,\ldots,\alpha_n\in\Phi_i$ be the simple roots.

  For any prime power $p^e$ let $\fC[p^e]\subseteq\fC_\cP$ be the
  kernel of multiplication by $p^e$. The union $\fC[p^\infty]$ over
  all $e$ is the $p$-primary component of $\fC_\cP$. Since $\cC_\cP$
  is the direct sum of its primary components it suffices to show the
  assertion for $\fC[p^\infty]$. Now it follows from
  \cref{lem:infiniteC} that $\fC[p^e]/\fC[p^{e-1}]$ is a constant
  sheaf with stalks $\ZZ/p\ZZ$ which is supported in the face
  \[
    \{x\in\cP\mid\alpha_i(x)=0\text{ for all $i$ with }p^e\nmid a_i\}
    \qedhere\hfill\qed
  \]
\end{proof}

Since constant sheaves on contractible spaces have trivial cohomology,
we get:

\begin{corollary}\label{cor:Cvanish}

  Assume $\Lambda$ is of adjoint type. Then $H^i(\cP,\fC_\cP)=0$ for
  all $i\ge1$.

\end{corollary}

Next we study the kernel of $\epsilon$
(eq.~\eqref{eq:epsilon}). Observe that the constant sheaf $\RR_\cP$ is
contained in this kernel. From this we get a homomorphism
\[
  \bar\epsilon:\cC_\cP^W/\RR_\cP\to\fL_{\cP,\Lambda}^\Phi
\]

\begin{lemma}\label{L4}

  \label{L4i3} Let $\fK_\cP$ be the kernel of $\bar\epsilon$. Then its
  stalk in $x\in\cP$ is equal to $\Lambda^\vee\cap(\RR\VPhi_x)^\vee$.

\end{lemma}

\begin{proof}

  Let $x\in\cP$ and $U\subseteq\cP$ a small convex open
  neighborhood. A function $f$ on $U$ in the kernel of $\epsilon$ if
  and only if its gradient is in $\Gamma^\vee$. Continuity implies
  that $\nabla f$ must be in fact constant. This implies that $f$ is
  an affine linear function with $\fq\in\Gamma^\vee$. Moreover, $f$ is
  a section of $\cC_\cP^W$ if and only if $\fq$ is
  $W_x$-invariant. This means, $\fq$ should be orthogonal to all
  $\Valpha\in\VPhi_x$.
\end{proof}

\begin{lemma}

  Let $\Lambda$ be of adjoint type and let $\alpha_1,\ldots,\alpha_n$
  be the simple roots of $\Phi$. Let $H_i$ be the hyperplane
  $\{\alpha_i=0\}$. Then the sheaf $\fK_\cP$ fits into an exact
  sequence
  \[
    0\to\fK_\cP\to\Lambda^\vee_\cP\overset\rho\to
    \bigoplus_{i=1}^n\ZZ_{H_i\cap\cP}\overset\psi\to\fC_\cP\to0.
  \]

\end{lemma}

\begin{proof}

  All sheaves are clearly restrictions of the corresponding sheaves on
  $\cA$ to $\cP$. Thus we may assume that $\cP=\cA$. Then we may treat
  every factor of the root system $\Phi$ separately. Thus, we may
  assume that $\Phi$ is either finite or irreducible and infinite.

  For $x\in\cP$ we have to show that the stalk
  $\fK_x=\Lambda^\vee\cap(\RR\VPhi_x)^\vee$ fits into an exact
  sequence
  \[\label{eq:kernelsequence}
    0\to\fK_x\to\Lambda^\vee\overset{\rho_x}\to
    \ZZ^{I_x}\overset{\psi_x}\to\fC_x\to0.
  \]
  First, we define $\rho_x$ as
  $\rho_x(v):=(\<\Valpha_i,v\>)_{i\in I_x}\in\ZZ^{I_x}$. Then $\fK_x$
  is clearly the kernel of $\rho_x$.

  If $\Phi$ is finite then the set of all $\Valpha_i$ with $i\in I_x$
  is part of a dual basis of $\Lambda^\vee$. Thus $\rho_x$ is
  surjective. Thus \eqref{eq:kernelsequence} is exact since $\fC_x=0$
  in this case.

  Now assume that $\Phi$ is irreducible and infinite. Then
  $\fC_x=\ZZ/d_{I_x}\ZZ$ and we define $\psi_x$ as
  $\psi_x(y_i):=\sum_{i\in I_x}a_iy_i+d_{I_x}\ZZ$. Identifying $\Vfa$
  with the hyperplane $H$ as in the proof of \cref{lem:infiniteC} we
  have to show that
  \[
    \ZZ^n\cap H\overset{\rho_x}\to
    \ZZ^{I_x}\overset{\psi_x}\to\ZZ/d_I\ZZ\to0
  \]
  is exact. Surjectivity follows again from
  $\operatorname{gcd}(d_I,d_{I'})=1$. Moreover, the kernel of $\psi_x$
  consists of all $(y_i)_{i\in{I_x}}$ which can be extended to an
  $n$-tuple $(y_i)_{i=1}^n\in\ZZ^n$ with $\sum_{i=1}^na_iy_i=0$. This
  is equivalent to $\sum_{i\in I_x}a_iy_i$ being divisible by $d_I$.
\end{proof}

\begin{lemma}

  Assume $\Lambda$ to be of adjoint type. Then the homomorphism
  \[
    H^0(\psi):H^0(\cP,\bigoplus_i\ZZ_{H_i\cap\cP})\to H^0(\cP,\fC_\cP)
  \]
  is surjective.

\end{lemma}

\begin{proof}

  Both sides decompose as direct sums according to the decomposition
  of $\Phi$ into factors. Thus we may assume that $\Phi$ is
  irreducible. Then there is nothing to prove if $\Phi$ is finite
  since then $\fC_\cP=0$. So assume that $\Phi$ is infinite.

  Let $p$ be a prime. Then it suffices to show $H^0(\psi)$ is
  surjective on $p$-primary components. For $e\ge0$ let
  $\cF_e\subseteq\cA$ be the face of $\cA$ which is defined by the
  equations $\alpha_i=0$ whenever $p^e\nmid a_i$. Then
  $\cA=\cF_0\supseteq\cF_1\supseteq\ldots$. Let $e$ be maximal with
  $\cF_e\cap\cP\ne\leer$. Then the explicit description of $\cC$
  (\cref{lem:infiniteC}) and the convexity of $\cP$ yield that
  $H^0(\cP,\cC)[p^\infty]=\ZZ/p^e\ZZ$. We may assume that
  $e\ge1$. Then there is a simple root $\alpha_{i_0}$ with
  $p\nmid a_{i_0}$. Since then $\alpha_{i_0}(x)=0$ for all
  $x\in\cF_e\cap\cP$, the summand $\ZZ_{H_{i_0}\cap\cP}$ contributes
  to a summand $\cong\ZZ$ in
  $H^0(\cP,\bigoplus_i\ZZ_{H_i\cap\cP})$. Moreover, the restriction of
  $H^0(\psi)$ to this summand is multiplication by $a_{i_0}$ followed
  by reducing mod $p^e$. This yields the assertion.
\end{proof}

\begin{lemma}\label{lem:Kvanish}

  Assume $\Lambda$ is of adjoint type. Then $H^i(\cP,\fK_\cP)=0$ for
  all $i\ge2$.

\end{lemma}

\begin{proof}

  Let $\fT$ be the kernel of $\psi$, yielding a short exact sequence
  \[
    0\to\fT\to\bigoplus_i\ZZ_{H_i\cap\cP}\overset\psi\to\fC_\cP\to0.
  \]
  Since $H_i\cap\cP$ is convex, hence contractible, the higher
  cohomology of $\bigoplus_i\ZZ_{H_i\cap\cP}$ vanishes. We already
  proved that $H^i(\cP,\fC)=0$ for all $i\ge1$. Combined with the
  surjectivity of $H^0(\psi)$ this implies that $H^i(\cP,\fT)=0$ for
  all $i\ge1$. Since also $H^i(\cP,\Lambda^\vee)=0$ for all $i\ge1$,
  the short exact sequence
  \[
    0\to\fK\to\Lambda_\cP^\vee\to\fT\to0
  \]
  implies $H^i(\cP,\fK_\cP)=0$ for all $i\ge2$.
\end{proof}

\begin{proof}[Proof of Theorem \ref{T1}]

  By \cref{lem:commensurable} we may assume that $\Lambda$ is of
  adjoint type. Consider the short exact sequence
  \[
    0\to\RR_\cP\to\cC_\cP^W\to\cC_\cP^W/\RR_\cP\to0.
  \]
  Since $\cC_\cP^W$ is a soft sheaf and $\cP$ is contractible, the
  higher cohomology of all three sheaves vanishes. Let
  $\fS\subseteq\fL_{\cP,\Lambda}^\Phi$ be the image of
  $\epsilon$. Then we get a short exact sequence
  \[
    0\to\fK_\cP\to\cC_\cP^W/\RR_\cP\to\fS\to0.
  \]
  \cref {lem:Kvanish} implies $H^i(\cP,\fS)=0$ for all
  $i\ge1$. Finally, \cref{cor:Cvanish} and the short exact sequence
  \[
    0\to\fS\to\fL^\Phi_{\cP,A}\to\fC_\cP\to0
  \]
  imply $H^i(\cP,\fL^\Phi_{\cP,\Lambda})=0$ for all $i\ge1$.
\end{proof}


\section{Examples}\label{sec:examples}

We conclude this paper with a series of examples which is in no way
comprehensive. It should be mentioned that Paulus has obtained in his
thesis, \cite{Paulus}, more complete results by classifying certain
subclasses of multiplicity free manifolds. Some of the examples below
are in fact due to him.

Given a spherical pair $(\cP,\Lambda)$, how do we construct $M$ with
$(\cP_M,\Lambda_M)=(\cP,\Lambda)$?  In section \ref{sec:class}, this
problem was essentially reduced to results from \cite{KnopAuto}, which
we wish to state more explicitly.

Choose any point $x\in\cP$ and let $L\subseteq K$ be the centralizer
of $\exp(x)\in K$. Then, by assumption, there is a smooth affine
spherical $L_\CC$-variety $X$ with $\Lambda_X^+=C_x\cP\cap\Lambda$
where $\Lambda_X^+$ is the set of highest weights occurring in the
coordinate ring of $X$. We call $X$ the \emph{local model of
  $(\cP,\Lambda)$ in $x$}. By Losev's theorem,
\cite{Losev}*{Thm.~1.3}, it is unique. Let $U\subseteq\cP$ be a small
open convex neighborhood of $x$. Then it essentially a result of Brion
\cite{Brion} and Sjamaar \cite{Sjamaar} that the cross-section
$M_0=\log M_U$ (see \eqref{eq:MU}) is isomorphic, as an
$L$-Hamiltonian manifold, to an open subset $X_0\subseteq X$ where the
Hamiltonian structure on $X$ is induced by a closed embedding into a
unitary $L$-representation. It follows that $m^{-1}(U)$ is isomorphic
to the fiber product $K\times^L X_0$. The latter is an open subset of
$K\times^LX$ which we therefore call the \emph{big local model at
  $x$}.

Using Remark \emph{ii)} after \cref{thm:cross}, there is even a
canonical (maximal) choice for $U$, namely $\cP_x:=\cP\setminus F$
where $F$ is the union of all (closed) faces of $\cP$ which do not
contain $x$. Since $\cP$ is locally polyhedral it is easily seen that
$\cP_x\subseteq\cP$ is open and that $x$ lies in the unique closed
face of $\cP_x$. Let $\cA^\sigma\subseteq\cA$ be the smallest face of
the alcove $\cA$ containing $x$. Let $\cA_\sigma\subseteq\cA$ be as in
\eqref{eq:openstar}. Then clearly $\cP=\cP_0\subseteq\cA_\sigma$. This
entails that $M_U$ is globally the ``exponential'' of Hamiltonian
$L$-variety $(M_0,m_0)$ with $(m_0)_+(M_0)=\cP_x-x$.

The local model $X$ always has the form $X=L_\CC\times^{H_\CC}V$ where
$H\subseteq L$ is a closed subgroup and $V$ is a (complex)
representation of $H_\CC$ (see \cite{KVS}*{Cor.~2.2}). The variety $X$
contains a unique closed $L_\CC$-orbit namely $L_\CC/H_\CC$. This
implies that the $L$-orbit $L/H$ is isomorphic to the fiber
$m^{-1}(x)\subseteq M$. Hence, the fiber of $m_+$ over $x$ is $K/H$.

The above gives a way to check whether a given pair is spherical. In
fact, by work of Pezzini-Van Steirteghem, \cite{PVS}, this becomes a
completely combinatorial problem. Constructing spherical pairs is more
difficult, though. A frequently successful strategy for finding
spherical pairs is to start with a local model $X$. To this end, Van
Steirteghem and the author have compiled a comprehensive list of all
smooth affine spherical varieties, see \cite{KVS}. The corresponding
weight monoids $\Lambda_X^+$ are calculated in the forthcoming paper
\cite{KPVS}.  This then yields the tangent cone $\cC_x\cP$ and the
lattice $\Lambda$.

1. Disymmetric spaces

Nevertheless, we start with examples which are not obtained by the
method detailed above. Recall that a symmetric space is a homogeneous
space of the form $K/K^\tau$ where $K^\tau$ is the fixed point group
of an involution $\tau$ of $K$. The product
$K/K^\sigma\times K/K^\tau$ of two symmetric spaces together with the
diagonal $K$-action will be called a \emph{disymmetric space}. They
form a very important class of multiplicity free manifolds:

\begin{theorem}

  Let $K$ be a connected compact Lie group with two involutions
  $\sigma$, $\tau$ and let $M=K/K^\sigma\times K/K^\tau$ be the
  corresponding disymmetric space. Then $M$ carries the structure of a
  \mf\ $K\sigma\tau$-manifold such that the moment map is
  \[\label{eq:disym1}
    m:M\to K\sigma\tau: (aK^\sigma,bK^\tau)\mapsto a\,({}^\sigma
    a^{-1})({}^\sigma b)({}^{\sigma\tau}b^{-1})
  \]

\end{theorem}

\begin{proof}

  According to \cite{AMM}*{3.1} any conjugacy class is a \qH\ manifold
  with the inclusion into the group being the moment map. Apply this
  to the group $\ZZ\sigma\ltimes K$. Then the conjugacy class of
  $\sigma$ can be identified with $K/K^\sigma$. Hence, the symmetric
  space $K/K^\sigma$ is a \qH\ $K\sigma$-manifold with moment map
  \[
    K/K^\sigma\to K\sigma:aK^\sigma\mapsto a\,({}^\sigma a^{-1})
  \]
  Analogously, the space $K/K^\tau$ is a \qH\ $K\tau$-manifold. Hence
  the fusion product $M=K/K^\sigma\otimes K/K^\sigma$ (see \cite{AMM}
  or proof of \cref{prop:local-qH}) is \qH\ with twist
  $\sigma\tau$. The moment map of $M$ is just the product of the two
  moment maps of the factors which amounts to formula
  \eqref{eq:disym1}.

  So far, we have not even used that $\sigma$ and $\tau$ are
  involutions. This only comes in to show that $M$ is multiplicity
  free. From \eqref{eq:disym1} we get that the derivative of $m$ at
  $z:=(1K^\sigma,1K^\tau)\in M$ is
  \[\label{eq:Dmoment}
    Dm:\fk/\fk^\sigma\times\fk/\fk^\tau\to\fk:(\xi,\eta)\mapsto
    \xi-{}^\sigma\xi+{}^\sigma\eta-{}^{\sigma\tau}\eta.
  \]
  Now let $\fp_\sigma,\fp_\tau\subseteq\fk$ be the $(-1)$-eigenspaces
  of $\sigma$, $\tau$, respectively. Then
  $\fk/\fk^\sigma\times\fk/\fk^\tau\cong\fp_\sigma\times\fp_\tau$ and
  \eqref{eq:Dmoment} becomes
  \[
    Dm:\fp_\sigma\times\fp_\tau\to\fk:(\xi,\eta)\mapsto 2\xi+2\
    {}^\sigma\eta=2\ {}^\sigma(\eta-\xi).
  \]
  Thus, $\ker Dm$ equals $\Fq:=\fp_\sigma\cap\fp_\tau$ which is
  embedded diagonally into $\fp_\sigma\times\Fq_\tau$. For $\xi\in\Fq$
  and $t\in\RR$ we have
  \[
    e^{t\xi}\cdot m(z)=e^{t\xi}\cdot1\cdot e^{-\sigma\tau(t\xi)}=1.
  \]
  Hence $\Fq m(z)=0$, i.e., $\Fq$ is parallel to the fiber
  $M_z:=m^{-1}(m(z))$. On the other hand $\Fq z=\Fq=\ker Dm$ which
  implies that $K_{m(z)}$ acts transitively on $M_z$. There was
  nothing particular about the point $z$ since every point
  $(aK^\sigma,bK^\tau)\in M$ has the form of $z$ after replacing
  $\sigma$, $\tau$ by their conjugates with $a$, $b$,
  respectively. This shows that the symplectic reductions
  $M_z/K_{m(z)}$ are discrete for all $z\in M$, i.e., $M$ is
  multiplicity free.
\end{proof}

Of special importance is the case when $\sigma$ and $\tau$ represent
the same element in $\|Out|(K)$. Then $\sigma\tau$ is inner and we
get:

\begin{corollary}

  Let $K$, $\sigma$, $\tau$ and $M$ as above. Assume that there is an
  element $u\in K$ with $\sigma\tau=\Ad(u)$. Then $M$ carries the
  structure of an (untwisted) \mf\ $K$-manifold such that the moment
  map is
  \[
    m:M\to K: (aK^\sigma,bK^\tau)\mapsto a({}^\sigma a^{-1})\ u\
    ({}^\tau b)b^{-1}
  \]

\end{corollary}

\begin{proof}

  Follows from (see \eqref{eq:disym1})
  \begin{multline}
    m(aK^\sigma,bK^\tau)=a\sigma a^{-1}b\tau b^{-1}=\\= a\cdot\sigma
    a^{-1}\sigma\cdot\sigma\tau\cdot\tau b\tau\cdot b^{-1}
    =a({}^\sigma a)^{-1}u\ ({}^\tau b)b^{-1}
  \end{multline}
\end{proof}

The orbit structure of disymmetric spaces has been thoroughly
investigated by Matsuki in a series of papers
\cites{Matsuki1,Matsuki2,Matsuki3}. In particular, he showed that the
orbit space $M/K$ can be identified with a polytope which is of course
our $\cP_M$. Its shape is controlled by a Weyl group which is closely
related to our $W_M$. In fact, it is possible to derive our invariants
$\cP_M$, $\Lambda_M$, and $\Phi_M$ from Matsuki's calculations. Since
the results do not really fit into the present paper, details will
appear elsewhere.

Nevertheless, to show the idea, we will give one instructive example
without proofs, namely where $K=SU(2n)$ with $n\ge2$ and $\sigma$ and
$\tau$ are the involutions defining the subgroups $K^\sigma=SO(2n)$
and $K^\tau=Sp(2n)$, respectively. It is well known that
\[
  \cA=\{(x_1,\ldots,x_{2n})\mid x_1\ge\ldots\ge x_{2n}\ge x_1-1\text{
    and }x_1+\ldots+x_{2n}=0\}
\]
is the fundamental alcove for $SU(2n)$. The simple roots are
\[
  \alpha_1:=x_1-x_2,\ldots, \alpha_{2n-1}=x_{2n-1}-x_{2n},
  \alpha_{2n}=x_{2n}-x_1+1=\alpha_0
\]
It follows from Matsuki's calculations that the momentum image of $M$
is
\[
  \cP_M=\bigg\{(y_1+\viertel,\ldots,y_n+\viertel,
  y_1-\viertel,\ldots,y_n-\viertel) \left|
    \begin{matrix}y_1\ge\ldots\ge y_n\ge y_1-\half\\
      y_1+\ldots+y_n=0\end{matrix}\right\}
\]
which implies
\[
  \fa_M=\{x_1-x_{n+1}=\ldots=x_n-x_{2n}=\half,\ x_1+\ldots+x_{2n}=0\}
\]
The simple roots of $\Phi_M$ are
\[
  \sigma_1=\alpha_1+\alpha_{n+1},\ldots,\sigma_n=\alpha_n+\alpha_{2n}.
\]
which shows that $\Phi_M$ is an affine root system of type
$\sA_{n-1}^{(1)}$.  The character group $\Lambda_M$ of $M$ is just the
weight lattice of $\VPhi_M$ in $\Vfa_M$. The point $x\in\cP_M$ with
coordinates $y_1=\ldots=y_n=0$ is a vertex of $\cP_M$. Since all
simple roots of $K$ except for $\alpha_0$ and $\alpha_n$ vanish in $x$
the centralizer of $\exp(x)$ in $K$ is the Levi subgroup
$L=S(U(n)\times U(n))$. Its complexification is
$L_\CC=S(GL(n,\CC)\times GL(n,\CC))$. So the local model $X$ in $x$
must have the weight monoid $\Lambda_X^+=\cC_x\cP_M\cap\Lambda_M$. Let
$H=GL(n,\CC)$ be embedded into $L_\CC$ via
$g\mapsto(g,(g^t)^{-1})$. Then one checks that $L_\CC/H$ has indeed
the desired weight monoid and therefore equals $X$.  Observe that
$X=SL(n,\CC)$ with $L_\CC$-action $(g_1,g_2)g=g_1gg_2^t$. The big
local model in $x$ is
\[
  Z:=SU(2n)\times^LX.
\]
The other vertices are all conjugate under the center of $K$ and
therefore have isomorphic local models. The upshot is that
$M=SU(2n)/SO(2n)\times SU(2n)/Sp(2n)$ can be obtained by gluing $n$
copies of $Z$ but in a non-holomorphic way.

Another instance of a disymmetric manifold is the double $D(K_0)$ of a
connected compact Lie group (see also the proof of
\cref{prop:local-qH}). Put $K:=K_0\times K_0$ and let $\tau$ be the
switching automorphism. Then $K^\tau$ is the diagonal in $K$ and
$K/K^\tau=K_0$ is a $\tau$-twisted conjugacy class. Therefore,
$D(K_0)=K_0\times K_0$ is a \mf\ $K$-manifold with diagonal action
\[
  (x,y)*(a,b)=(xay^{-1},xby^{-1})
\]
and moment map
\[
  m(a,b)=(ab^{-1},a^{-1}b).
\]
After the coordinate change
\[
  D(K_0)\to D(K_0):(a,b)\mapsto(a,b^{-1})
\]
this is precisely the double in the sense of \cite{AMM}*{\S3.2}. Let
$\cA_0$ and $P_0$ be the alcove and the weight lattice of $K_0$. Then
$\cA=\cA_0\times\cA_0$ and $P=P_0\oplus P_0$ are alcove and weight
lattice of $K$. Let, moreover, $\delta(\chi):=-w_0(\chi)$ be the
opposition endomorphism of $\fa_0$ where $w_0$ is longest element of
the Weyl group of $K_0$. Then
\[
  (\cP_{D(K_0)},\Lambda_{D(K_0)})=
  (\id\times\delta)(\cA_0,P_0)\subseteq(\cA,P)
\]
are the data of $D(K_0)$.

\begin{remark}

  Already the case $M=D(SU(2))$ exhibits a new phenomenon namely that
  the classification of multiplicity free manifolds depends on the
  choice of the invariant scalar product on $\fk$. To see this,
  observe that the alcove $\cA$ for $K=SU(2)\times SU(2)$ is a
  rectangle whose side lengths depend on the metric. Let $\cP$ be a
  diagonal of $\cA$. Then in order to be spherical, $\cP$ has to be
  parallel to the sum $\alpha+\alpha'$ of the simple roots. This holds
  if and only if $\cA$ is a square, i.e., the metrics on both factors
  of $K$ are the same.

\end{remark}

2. Groups of rank $1$

Let $K=SU(2)$. Then $\cA$ is an interval and $\cP\subseteq\cA$ is a
subinterval. If $\cP_M\ne\cA$ then the discussion in Remark \emph{ii)}
after \cref{thm:cross} shows that $M$ is the exponential of a
Hamiltonian manifold. Quasi-Hamiltonian manifolds which are not of
this form will be called \emph{genuine}. So a genuine \mf\
$SU(2)$-manifold necessarily has $\cP_M=\cA$. The possible local
models in the end points are the $SL(2,\CC)$-varieties $\CC^2$,
$SL(2,\CC)/\CC^*$ and $SL(2,\CC)/N(\CC^*)$. Accordingly, there are $3$
different genuine \mf\ $SU(2)$-manifolds namely:

\begin{itemize}

\item $\Lambda_M=P$ ($\cong\ZZ$, the weight lattice of $SU(2)$). Here
  $M$ is obtained by gluing two copies of $\CC^2$ and is therefore
  diffeomorphic to the $4$-sphere $S^4$. This example has been found
  by by Alekseev-Meinrenken-Woodward \cite{AMW} under the name
  ``spinning $4$-sphere''.

\item $\Lambda_M=2P$. Then $M$ is the disymmetric manifold
  $M=S^2\times S^2$.

\item $\Lambda_M=4P$. In this case, $M$ is the quotient of the
  previous one by the switching involution. Hence $M\cong \P^2(\CC)$.

\end{itemize}

There is another affine root system of rank $1$ namely $\sA_2(2)$. It
belongs to $K=SU(3)$ and an outer automorphism $\tau$ of $K$,
e.g. complex conjugation. The alcove $\cA$ is an interval and the two
roots $\alpha_0$, $\alpha_1$ satisfy $\Valpha_0=-2\Valpha_1$. The
weight lattice is $P=\ZZ\Valpha_1$. The centralizers corresponding to
the end points are $SU(2)$ and $SO(3)$, respectively. Let
$\cP_M=\cA$. Then a discussion as above yields two cases

\begin{itemize}

\item $\Lambda_M=P$: In this case, the local models are $\CC^2$ and
  $SO(3,\CC)/SO(2,\CC)$.

\item $\Lambda_M=2P$. In this case, the local models are
  $SL(2,\CC)/SO(2,\CC)$ and $SO(3,\CC)/O(2,\CC)$.

\end{itemize}

Note that $\Lambda_M=4P$ does not work since
$\Lambda_X^+=\ZZ_{\ge0}(4\Valpha_1)$ is not the weight monoid of any
\emph{smooth} $SO(3,\CC)$-variety.

3. Manifolds of rank $1$

The spinning $4$-sphere has been generalized in \cite{HJS} by
Hurtubise-Jeffrey-Sjamaar to that of a spinning $2n$-sphere. In our
terms it can be constructed as follows: let $K=SU(n)$. Then the alcove
$\cA$ has $n$ vertices namely $x_0=0$ and the fundamental weights
$x_i=\omega_i$, $i=1,\ldots,n-1$. Let $\cP$ be the edge joining $x_0$
and $x_1$. Let $\Lambda=\ZZ\omega_1$. Then $(\cP,\Lambda)$ with
$\Lambda=\ZZ\omega_1$ is a spherical pair. Indeed, the weights of the
smooth affine spherical $SL(n,\CC)$-variety $X=\CC^n$ form the monoid
$\ZZ_{\ge0}\omega_1$. This shows that it is a local model at the
vertex $x_0$. The situation in $x_1$ is similar: the centralizer is
still $L=K=SU(n)$ but the simple root system is different, namely
$\alpha_2,\alpha_3,\ldots,\alpha_{n-1},\alpha_n=\alpha_0$. The last
fundamental weight with respect to this system is
$-\omega_1$. Therefore the monoid
$C_{x_1}\cP\cap\Lambda=\ZZ_{\ge0}(-\omega_1)$ has a model, as well,
namely again $\CC^n$. Glued together this yields the spinning
$2n$-sphere.

Eshmatov, \cite{Eshmatov}, has found an analogue of the spinning
$2n$-sphere for the symplectic group. More precisely, he showed that
$M=\P(\HH^n)$ carries the structure of a \mf\ $Sp(2n)$-manifold. Using
our theory, this structure can be obtained as follows. Let $K=Sp(2n)$
and let $\epsilon_1,\ldots,\epsilon_n$ be the standard basis of the
Cartan subalgebra $\ft$. Let $\cP$ be the line segment joining the
origin $x_0=0$ with $x_1=\half\epsilon_1$. This is an edge of the
fundamental alcove $\cA$. Put $\Lambda:=\ZZ\epsilon_1$. Then
\[
  X_1=\CC^{2n}
\]
is an (even big) local model in $x_0$. The other endpoint $x_1$
behaves differently, though. In this case the simple roots of $L$ are
$\alpha_0,\alpha_2,\alpha_3,\ldots,\alpha_n$ which yields
$L=Sp(2)\times Sp(2n-2)$. Moreover $-\omega_1$ is now the fundamental
weight of the first factor of $L$. Thus, the smooth affine spherical
variety with character group $\ZZ_{\ge0}(-\omega_1)$ is simply $\CC^2$
with the second factor of $L$ acting trivially. As a big local model
at $x_1$ we obtain
\[
  X_2=Sp(2n)\Times^{Sp(2)\times Sp(2n-2)}\CC^2.
\]
Now Eshmatov's space is obtained by gluing $X_1$ and $X_2$.

This example has been further generalized by Paulus. We keep
$K=Sp(2n)$. Then the vertices of $\cA$ are
$x_k:=\sum_{i=1}^k\epsilon_k$ for $k=0,\ldots,n$. Fix $k$ with $k>0$
and let $\cP_k$ be the line segment joining $x_{k-1}$ and $x_k$. Let
moreover $\Lambda_k:=\ZZ\epsilon_k$. Then one shows as above that
$(\cP_k,\Lambda_k)$ is spherical and it is even possible to identify
the corresponding manifold:

\begin{theorem}

  Let $n,k$ be integers with $1\le k\le n$. Then there is a \mf\
  $Sp(2n)$-manifold structure on the quaternionic Grassmannian
  $M=\Gr_k(\HH^{n+1})$ with $(\cP_M,\Lambda_M)=(\cP_k,\Lambda_k)$.

\end{theorem}

\begin{proof}

  The big local models at $x_{k-1}$ and $x_k$, respectively are the
  spaces
  \[
    X_1:=Sp(2n)\times^{H_{k-1}}\CC^{2n-2k+2}\text{ and }
    X_2:=Sp(2n)\times^{H_k}\CC^{2k}
  \]
  where $H_k:=Sp(2k)\times Sp(2n-2k)\subseteq Sp(2n)$ and they glue to
  a \mf\ manifold $M$. Now recall that $Sp(2n)$ can also be
  interpreted as the unitary group of $\HH^n$. Then $H_k$ is the
  isotropy group of $\HH^k\subseteq\HH^n$. Therefore $X_2$ can be
  identified with the universal bundle $\widetilde{\Gr}_k(\HH^n)$ over
  the quaternionic Grassmannian $\Gr_k(\HH^n)$. Similarly, $X_1$ is
  isomorphic to $\widetilde{\Gr}_{n-k+1}(\HH^n)$. Now consider the
  space $\HH^{n+1}=\HH^n\oplus\HH$ where $K$ acts on the first
  factor. Let $e:=(0,1)$ be the fixed point. Each element of
  $\widetilde{\Gr}_k(\HH^n)$ can be interpreted as a pair $(L,v)$ with
  $L\in\Gr_k(\HH^n)$ and $v\in L$. Let
  $\Gamma_{L,v}\subseteq\HH^n\oplus\HH$ be the graph of the map
  $L\to\HH:u\mapsto\<u,v\>$. Then the map $(L,v)\mapsto\Gamma_{L,v}$
  identifies $X_2=\widetilde{\Gr}_k(\HH^n)$ with the open subset of
  all $\tilde L\in\Gr_k(\HH^{n+1})$ with $e\not\in \tilde
  L$. Similarly, $X_1$ be identified with the set of all
  $\tilde L\in\Gr_k(\HH^{n+1})$ with $e\not\in\tilde L^\perp$. So
  $\Gr_k(\HH^{n+1})$ is also obtained by gluing $X_1$ and $X_2$. It is
  easy to see that there is only one $K$-equivariant way to do that,
  so $M\cong\Gr_k(\HH^{n+1})$ (see e.g.\ \cite{GWZ}).
\end{proof}

\begin{remark}

  The complex Grassmannians $\Gr_k(\CC^{n+1})$ are \mf, as well,
  namely for $K=SU(n)$. But they are not genuine, i.e., they are
  ``exponentials'' of (ordinary) Hamiltonian manifolds.

\end{remark}

5. Surjective moment maps

It is interesting to look at \mf\ manifolds $M$ which are in a sense
as big as possible. That means, first of all, that $\cP_M$ is the
entire alcove. This is clearly equivalent to the moment map being
surjective. There are quite a few of them, most of which are
disymmetric. Therefore we also demand that $\Lambda_M$ be as big as
possible, i.e., equals the weight lattice $P$ of $\VPhi$. This is
equivalent to the $K$-action on $M$ being free.

\begin{theorem}

  Let $(K,\tau)$ be one of the following three cases:
  \[
    (SU(n),\id),\quad(Sp(2n),\id),\quad (SU(2n+1),k\mapsto\overline k)
  \]
  (the last $\tau$ is complex conjugation). Then $(\cA,P)$ is
  spherical, i.e., there is a unique \mf\ $K\tau$-manifold $M$ whose
  moment map is surjective and such that $K$ acts freely on $M$.

\end{theorem}

\begin{proof} It suffices to find a local model in each of the
  vertices $x$ of $\cA$. For that, each case will be treated
  separately.

  $(K,\tau)=(SU(n),\id)$: We start with $x=0\in\cP=\cA$. Then $L=K$
  and $C_x\cA$ is the dominant Weyl chamber. Therefore, we have to
  show that there is a smooth affine $SL(n,\CC)$-variety $X_n$ such
  that $\cO(X)=\bigoplus_\chi L(\chi)$ where $\chi$ runs through all
  dominant weights. Such a variety does in general not exist for an
  arbitrary reductive group but it does for $SL(n,\CC)$ namely
  \[
    X_n:=\begin{cases}
      SL(n,\CC)\Times^{Sp(n,\CC)}\CC^n&\text{if $n$ is even},\\
      SL(n,\CC)/Sp(n-1,\CC)&\text{if $n$ is odd}.
    \end{cases}
  \]
  Thus $(\cA,\Lambda)$ is spherical in $x=0$. But then it is also
  spherical in all other vertices since they differ only in a
  translation by an element of the center.

  $(K,\tau)=(Sp(2n),\id)$: A model in $x=0$ is
  \[
    Y_n:=\begin{cases}
      Sp(2n,\CC)\Times^{Sp(n,\CC)\times Sp(n,\CC)}\CC^n&\text{if $n$ is even},\\
      Sp(2n,\CC)\Times^{Sp(n-1,\CC)\times
        Sp(n+1,\CC)}\CC^{n+1}&\text{if $n$ is odd}.
    \end{cases}
  \]
  In general, $\cA$ has $n+1$ vertices $x_0=0,x_1,\ldots,x_n$ which
  are enumerated in such a way that $\alpha_i(x_i)\ne0$. Then the
  centralizer in $L$ in $x_i$ is $L=Sp(2i)\times Sp(2n-2i)$. Since
  $\Lambda=\ZZ^n=\ZZ^i\oplus\ZZ^{n-i}$ splits accordingly, the
  varieties
  \[
    Y_{i,n-i}:=Sp(2n)\Times^{Sp(2i)\times Sp(2n-2i)}(Y_i\times
    Y_{n-i})
  \]
  are the big local models in $x_i$.

  $(K,\tau)=(SU(2n+1),\tau)$ with $\tau$ an outer automorphism: Here,
  the Dynkin diagram of $(K,\tau)$ is of type $\sA_{2n}^{(2)}$ (see
  \eqref{eq:A2n1}). In this case, $\cA$ has $n+1$ vertices
  $x_0,\ldots,x_n$ such that the centralizer of $x_i$ is
  $L=Sp(2i)\times SO(2n+1-2i)$. It is well-known that the coordinate
  ring of
  \[
    Z_n:=SO(2n+1,\CC)/GL(n,\CC)
  \]
  contains all irreducible $SO(2n+1,\CC)$-modules exactly once. So
  \[
    Z_{i,n-i}:=SU(2n+1)\Times^{Sp(2i)\times SO(2n+1-2i)}(Y_i\times
    Z_{n-i})
  \]
  is a big local model in $x_i$.
\end{proof}

\begin{remarks}

  1. Using the spherical roots of the local models it is easy to
  determine the root system $\Phi_M$ in each case:

  Untwisted $SU(n)$: The simple affine roots of $K$ are
  \[
    \alpha_0=1+x_n-x_1,\alpha_1=x_1-x_2,\ldots,\alpha_{n-1}=x_{n-1}-x_n.
  \]
  The spherical roots of $X_n$ are
  $\alpha_1+\alpha_2,\alpha_2+\alpha_3,\ldots,\alpha_{n-2}+\alpha_{n-1}$. For
  $n$ odd, see \cite{BraviPezzini}. The even case is handled in
  \cite{KPVS} or, in this case, in \cite{Luna}. Therefore, the simple
  roots of $\Phi_M$ are
  \[
    1+x_n-x_2,x_1-x_3,x_2-x_4,\ldots,x_{n-2}-x_n,1+x_{n-1}-x_1.
  \]
  Hence
  \[
    \Phi_M\cong
    \begin{cases}
      \sA_{\frac n2-1}^{(1)}\times\sA_{\frac n2-1}^{(1)}&\text{$n$ even,}\\
      \vrule height 14pt width 0pt\sA_{n-1}^{(1)}&\text{$n$ odd.}\\
    \end{cases}
  \]
  Observe that in the odd case the root systems of $K$ and $M$ are
  isomorphic but they are not the same. For example, for $n=3$, i.e.,
  $K=SU(3)$, one gets this picture:
  
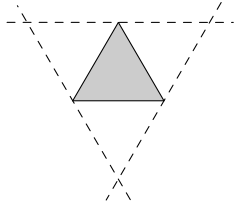
\[\label{eq:fig1}
    \begin{tikzpicture}[scale=0.2,line cap=round,line
      join=round,>=triangle 45,x=1.0cm,y=1.0cm]
      \fill[fill=black,fill opacity=0.2] (0.,0.) -- (6.,0.) --
      (3.,5.196152422706632) -- cycle; \draw (0.,0.)-- (6.,0.); \draw
      (6.,0.)-- (3.,5.196152422706632); \draw (3.,5.196152422706632)--
      (0.,0.); \draw [dash pattern=on 3pt off 3pt] (-4.32,5.2)--
      (10.56,5.196152422706632); \draw [dash pattern=on 3pt off 3pt]
      (9.762262054223738,6.516429029303968)--
      (2.19943667558704,-6.582768775266124); \draw [dash pattern=on
      3pt off 3pt] (3.785563324412959,-6.556788013152589)--
      (-3.6393197678831375,6.303486742963368);
    \end{tikzpicture}
  \]
  where the gray triangle denotes $\cP=\cA$ and the axes of the simple
  reflections of $\Phi_M$ are marked by dashed lines. There is also
  something to be observed in the even case: here all roots of
  $\Phi_M$ are perpendicular to the vector
  $\delta=(1,-1,\ldots,1,-1)\in\Vfa$. This means that the automorphism
  group of $M$ contains the $1$-dimensional torus
  $\exp(\RR\delta)\subseteq A$, i.e., the $SU(n)$-action on $M$
  extends to an $U(1)\times SU(n)$-action.

  $Sp(2n)$: The simple affine roots of $K$ are
  \[
    \alpha_0=1-2x_1,\alpha_1=x_1-x_2,\ldots,\alpha_{n-1}=x_{n-1}-x_n,\alpha_n=2x_n.
  \]
  The spherical roots of $Y_n$ are
  $\alpha_1+\alpha_2,\alpha_2+\alpha_3,\ldots,\alpha_{n-1}+\alpha_n$
  (see \cites{KPVS,Luna}). Therefore, the simple roots of $\Phi_M$ are
  \[
    1-x_1-x_2,x_1-x_3,x_2-x_4,\ldots,x_{n-2}-x_n,x_{n-1}+x_n.
  \]
  Hence
  \[
    \Phi_M\cong\sA_{n-1}^{(1)}
  \]
  For $K=Sp(4)$ one gets
  
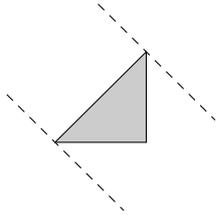
\[\label{eq:fig2}
    \begin{tikzpicture}[scale=0.2,line cap=round,line
      join=round,>=triangle 45,x=1.0cm,y=1.0cm]
      \fill[fill=black,fill opacity=0.2] (0.,0.) -- (6.,0.) -- (6.,6.)
      -- cycle; \draw (0.,0.)-- (6.,0.); \draw (6.,0.)-- (6.,6.);
      \draw (6.,6.)-- (0.,0.); \draw [dash pattern=on 3pt off 3pt]
      (-3.15,3.15)-- (4.51,-4.51); \draw [dash pattern=on 3pt off 3pt]
      (2.85,9.15)-- (10.51,1.49);
    \end{tikzpicture}
  \]
  Since there is again a $W_M$-invariant vector in $\Vfa$, namely
  $\delta=(1,-1,1,-1,\ldots)$ the $Sp(n)$-action on $M$ extends to an
  $U(1)\times Sp(n)$-action.

  Twisted $SU(2n+1)$: The simple affine roots of $K,\tau$ are
  \[
    \alpha_0=\half-2x_1,\alpha_1=x_1-x_2,\ldots,\alpha_{n-1}=x_{n-1}-x_n,\alpha_n=x_n.
  \]
  According \cite{BraviPezzini}, the spherical roots of $Z_n$ are
  $\alpha_1+\alpha_2,\alpha_2+\alpha_3,\ldots,\alpha_{n-1}+\alpha_n,\alpha_n$
  but here one has to be careful since the normalization of
  spherical roots is different form ours: in \cite{BraviPezzini} all
  roots are always primitive in the weight lattice while our roots are
  always ``as long as possible''. See \cite{BvS} for further
  details. This amounts to changing the last root $\alpha_n$ to
  $2\alpha_n$.  From this, we get the the simple roots of $\Phi_M$ as
  \[
    \half-x_1-x_2,x_1-x_3,x_2-x_4,\ldots,x_{n-2}-x_n,x_{n-1},2x_n
  \]
  Hence
  \[
    \Phi_M\cong\sA_{2n}^{(2)}.
  \]
  The fact that $\Phi_M$ has the same type as the root system of
  $(K,\tau)$ seems to be coincidental. For $K=SU(5)$ one gets
  
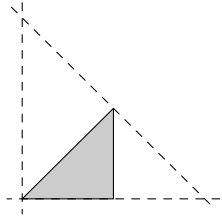
\[\label{eq:fig3}
    \begin{tikzpicture}[scale=0.2,line cap=round,line
      join=round,>=triangle 45,x=1.0cm,y=1.0cm]
      \fill[fill=black,fill opacity=0.2] (0.,0.) -- (6.,0.) -- (6.,6.)
      -- cycle; \draw (0.,0.)-- (6.,0.); \draw (6.,0.)-- (6.,6.);
      \draw (6.,6.)-- (0.,0.); \draw [dash pattern=on 3pt off 3pt]
      (0.,13.)-- (0.,-1.); \draw [dash pattern=on 3pt off 3pt]
      (-0.7071067811865475,12.707106781186548)--
      (12.70710678118655,-0.7071067811865497); \draw [dash pattern=on
      3pt off 3pt] (13.,0.)-- (-1.,0.);
    \end{tikzpicture}
  \]

  2. Paulus, \cite{Paulus}, worked out a complete list of \mf\
  manifolds (possibly twisted) with surjective moment map.

\end{remarks}

6. Inscribed triangles

For the last example, we were toying around with triangles inscribed
in an triangular alcove. Here are some examples of spherical pairs
$(\cP,\Lambda)$:
\[
  \begin{array}{|l|c|c|c|c|c|}
    \hline
    K&SU(3)&SU(3)&Sp(4)&Sp(4)&\sG_2\\
    \hline
    \begin{matrix}\cP\subseteq\cA\\\ \\\ \end{matrix}&\begin{tikzpicture}[scale=0.3,line cap=round,line join=round,>=triangle 45,x=1.0cm,y=1.0cm]
      \fill[fill opacity=0.2] (3.,0.) -- (4.5,2.598076211353316) --
      (1.5,2.598076211353316) -- cycle; \draw (0.,0.)-- (6.,0.); \draw
      (6.,0.)-- (3.,5.196152422706632); \draw (3.,5.196152422706632)--
      (0.,0.); \draw (3.,0.)-- (4.5,2.598076211353316); \draw
      (4.5,2.598076211353316)-- (1.5,2.598076211353316); \draw
      (1.5,2.598076211353316)-- (3.,0.); \draw
      (1.5,2.598076211353316)-- (3.,0.); \draw (3.,0.)--
      (4.5,2.598076211353316); \draw (4.5,2.598076211353316)--
      (1.5,2.598076211353316);
    \end{tikzpicture}
           &\begin{tikzpicture}[scale=0.3,line cap=round,line join=round,>=triangle 45,x=1.0cm,y=1.0cm]
             \fill[fill opacity=0.2] (4.,0.) --
             (4.,3.464101615137755) -- (1.,1.7320508075688774) --
             cycle; \draw (0.,0.)-- (6.,0.); \draw (6.,0.)--
             (3.,5.196152422706632); \draw (3.,5.196152422706632)--
             (0.,0.); \draw (4.,0.)-- (4.,3.464101615137755); \draw
             (4.,3.464101615137755)-- (1.,1.7320508075688774); \draw
             (1.,1.7320508075688774)-- (4.,0.); \draw
             (1.,1.7320508075688774)-- (4.,0.); \draw (4.,0.)--
             (4.,3.464101615137755); \draw (4.,3.464101615137755)--
             (1.,1.7320508075688774);
           \end{tikzpicture}
                 &\begin{tikzpicture}[scale=0.3,line cap=round,line join=round,>=triangle 45,x=1.0cm,y=1.0cm]
                   \fill[fill opacity=0.2] (3.,0.) -- (6.,3.) --
                   (3.,3.) -- cycle; \draw (0.,0.)-- (6.,0.); \draw
                   (6.,0.)-- (6.,6.); \draw (6.,6.)-- (0.,0.); \draw
                   (3.,0.)-- (6.,3.); \draw (6.,3.)-- (3.,3.); \draw
                   (3.,3.)-- (3.,0.); \draw (3.,0.)-- (6.,3.); \draw
                   (6.,3.)-- (3.,3.); \draw (3.,3.)-- (3.,0.);
                 \end{tikzpicture}
                       &\begin{tikzpicture}[scale=0.3,line cap=round,line join=round,>=triangle 45,x=1.0cm,y=1.0cm]
                         \fill[fill opacity=0.2] (4.,0.) -- (6.,2.)
                         -- (2.,2.) -- cycle; \draw (0.,0.)--
                         (6.,0.); \draw (6.,0.)-- (6.,6.); \draw
                         (6.,6.)-- (0.,0.); \draw (4.,0.)-- (6.,2.);
                         \draw (6.,2.)-- (2.,2.); \draw (2.,2.)--
                         (4.,0.); \draw (2.,2.)-- (4.,0.); \draw
                         (4.,0.)-- (6.,2.); \draw (6.,2.)-- (2.,2.);
                       \end{tikzpicture}
                             &\begin{tikzpicture}[scale=0.3,line cap=round,line join=round,>=triangle 45,x=1.0cm,y=1.0cm]
                               \fill[fill opacity=0.2] (4.,0.) --
                               (6.,1.1547005383792512) --
                               (4.,2.3094010767585025) -- cycle;
                               \draw (0.,0.)-- (6.,0.); \draw
                               (6.,0.)-- (6.,3.464101615137754);
                               \draw (6.,3.464101615137754)--
                               (0.,0.); \draw (4.,0.)--
                               (6.,1.1547005383792512); \draw
                               (6.,1.1547005383792512)--
                               (4.,2.3094010767585025); \draw
                               (4.,2.3094010767585025)-- (4.,0.);
                               \draw (4.,0.)--
                               (6.,1.1547005383792512); \draw
                               (6.,1.1547005383792512)--
                               (4.,2.3094010767585025); \draw
                               (4.,2.3094010767585025)-- (4.,0.);
                             \end{tikzpicture}\\
    \Lambda&P\text{ or }R&R&R&R&R\\
    \hline
  \end{array}
\]
We make no claim of completeness. In particular, we considered only
untwisted groups. The letters $P$ and $R$ denote the weight and the
root lattice of $K$, respectively. At each vertex, the complexified
centralizer $L$ is isogenous to $SL(2,\CC)\times\CC^*$. Then one can
show that the local models are either of the form $X=SL(2,\CC)/\mu_n$
in case $\cP$ touches $\cA$ in form of a reflection and
$X=SL(2,\CC)\times^{\CC^*}\CC$ if a right angle is involved.

\begin{remark}

  As communicated by Eckhart Meinrenken, the first case has also been
  found by Chris Woodward (unpublished).

\end{remark}

\end{document}